\documentclass[11pt]{article}

\usepackage[a4paper,left=23mm,right=23mm,top=25mm,bottom=25mm]{geometry}

\usepackage{amsmath,amsfonts,amssymb}
\usepackage{graphicx}
\usepackage{booktabs}
\usepackage[authoryear]{natbib}
\usepackage[plain,noend]{algorithm2e}
\usepackage{tikz}
\usetikzlibrary{arrows.meta,positioning}
\usepackage{subcaption}
\usepackage[shortlabels]{enumitem}
\usepackage{wrapfig}
\usepackage{comment}
\usepackage{authblk}     
\usepackage{hyperref}
\hypersetup{
  colorlinks=true,
  linkcolor=blue,
  citecolor=blue,
  urlcolor=blue
}

\usepackage{amsthm}
\newtheorem{theorem}{Theorem}
\newtheorem{definition}{Definition}
\newtheorem{lemma}{Lemma}
\newtheorem{assumption}{Assumption}  
\newtheorem{proposition}{Proposition}  
\newtheorem{corollary}{Corollary}  

\makeatletter
\renewcommand{\algocf@captiontext}[2]{#1\algocf@typo. \AlCapFnt{}#2}

\def\@algocf@capt@plain{top}
\renewcommand{\algocf@makecaption}[2]{%
  \addtolength{\hsize}{\algomargin}%
  \sbox\@tempboxa{\algocf@captiontext{#1}{#2}}%
  \ifdim\wd\@tempboxa > \hsize
    \hskip .5\algomargin
    \parbox[t]{\hsize}{\algocf@captiontext{#1}{#2}}%
  \else
    \global\@minipagefalse
    \hbox to\hsize{\box\@tempboxa}
  \fi
  \addtolength{\hsize}{-\algomargin}%
}
\makeatother

\newcommand{\indep}{\mathrel{\text{\scalebox{1.07}{$\perp\mkern-10mu\perp$}}}}

\title{Finite sample-optimal adjustment sets in linear Gaussian \\ causal models}

\author[1]{Nadja Rutsch\thanks{\,Corresponding author: \texttt{n.rutsch@vu.nl}}}
\author[2]{Sara Magliacane}
\author[1]{Stéphanie L. van der Pas}

\affil[1]{Department of Mathematics, Vrije Universiteit Amsterdam, The Netherlands}
\affil[2]{AMLab, Informatics Institute, University of Amsterdam, The Netherlands}

\date{\today} 

\begin{document}
\maketitle

\begin{abstract}
Traditional covariate selection methods for causal inference focus on achieving unbiasedness and asymptotic efficiency. In many practical scenarios, researchers must estimate causal effects from observational data with limited sample sizes or in cases where covariates are difficult or costly to measure. Their needs might be better met by selecting adjustment sets that are finite sample‑optimal in terms of mean squared error. In this paper, we aim to find the adjustment set that minimizes the mean squared error of the causal effect estimator, taking into account the joint distribution of the variables and the sample size. We call this finite sample‑optimal set the MSE‑optimal adjustment set and present examples in which the MSE‑optimal adjustment set differs from the asymptotically optimal adjustment set. To identify the MSE‑optimal adjustment set, we then introduce a sample size criterion for comparing adjustment sets in linear Gaussian models. We also develop graphical criteria to reduce the search space for this adjustment set based on the causal graph. In experiments with simulated data, we show that the MSE‑optimal adjustment set can outperform the asymptotically optimal adjustment set in finite sample size settings, making causal inference more practical in such scenarios.
\end{abstract}

\vspace{0.5cm}


\section{Introduction}

\begin{figure}
    \centering
       \begin{subfigure}[b]{0.21\textwidth}
        \centering
        \begin{tikzpicture}[
            font=\fontsize{8}{10}\selectfont,
            node distance=1.3cm,
            on grid,
            auto,
            block/.style={circle, draw, inner sep=0pt, outer sep=0pt, minimum size=0.5cm}
        ]
        
        \node [block] (A) {A};
        \node [block, above right=of A] (V2) {$W_2$};
        \node [block, right=of A, below right=of V2] (Y) {$Y$};
        \node [block, above right=of V2] (O1) {$O_1$};
        \node [block, above left=of V2] (V1) {$W_1$};         
        \node [block, below right=of A] (O2) {$O_2$};
        
        \draw[-{Latex[length=2mm]}, blue] (A) -- (Y) node[midway, above, black] {$\tau=3$}; 
        \draw[-{Latex[length=2mm]}] (V2) -- (O1) node[midway, right, pos=0.1] {$2$};           
        \draw[-{Latex[length=2mm]}] (V1) -- (O1)node[midway, above]{$2$};                   
        \draw[-{Latex[length=2mm]}] (O1) -- (Y) node[midway, right] {$5$};                 
        \draw[-{Latex[length=2mm]}] (V1) -- (A) node[midway, left] {$-1$};                    
        \draw[-{Latex[length=2mm]}] (V2) -- (A) node[midway, left, pos=0.1] {$0.1$};           
        \draw[-{Latex[length=2mm]}] (O2) -- (A) node[midway, left, pos=0.2] {$40$};           
        \draw[-{Latex[length=2mm]}] (O2) -- (Y) node[midway, right, pos=0.2] {$0.5$};        
                
\end{tikzpicture}
        \caption{$\mathcal{M}_1$}
    \end{subfigure}%
    \begin{subfigure}[b]{0.28\textwidth}
        \centering
        \includegraphics[width=0.95\textwidth]{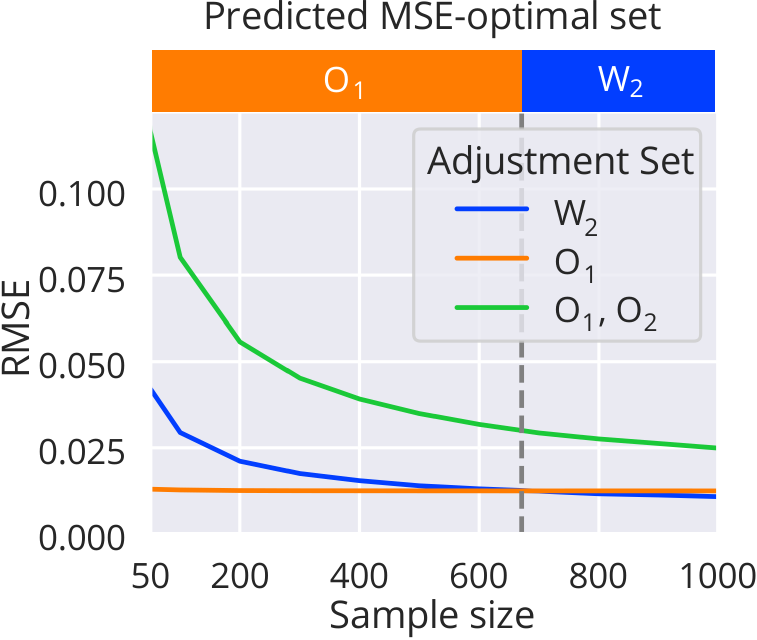}
        \caption{RMSE of $\hat \tau$ in $\mathcal{M}_1$}
    \end{subfigure}%
    \begin{subfigure}[b]{0.21\textwidth}
    \centering
        \begin{tikzpicture}[
            font=\fontsize{8}{10}\selectfont,
            node distance=1.3cm,
            on grid,
            auto,
            block/.style={circle, draw, inner sep=0pt, outer sep=0pt, minimum size=0.5cm}
        ]
        
        \node [block] (A) {A};
        \node [block, above right=of A] (M) {$C_1$};
        \node [block, below right=of M] (Y) {$Y$};
        \node [block, above=of Y, above right=of M] (O1) {$O_1$};
        \node [block, above=of A, above left=of M] (I1) {$W_1$};
        \node [block, below right=of A] (O2) {$O_2$};
        
        \draw[-{Latex[length=2mm]}, blue] (A) -- (Y) node[midway, above, black] {$\tau=0.29$};  
        \draw[-{Latex[length=2mm]}] (O1) -- (M) node[midway, left, pos=0.1] {$6$};             
        \draw[-{Latex[length=2mm]}] (I1) -- (M) node[midway, right, pos=0.1] {$1.33$};          
        \draw[-{Latex[length=2mm]}] (O1) -- (Y) node[midway, right] {$0.71$};                 
        \draw[-{Latex[length=2mm]}] (I1) -- (A) node[midway, left] {$0.55$};                  
        \draw[-{Latex[length=2mm]}] (O2) -- (A) node[midway, left, pos=0.2] {$1.1$};           
        \draw[-{Latex[length=2mm]}] (O2) -- (Y) node[midway, right, pos=0.2] {$0.14$};         
                
\end{tikzpicture}
        \caption{$\mathcal{M}_2$}
    \end{subfigure}
    \begin{subfigure}[b]{0.28\textwidth}
        \centering
        \includegraphics[width=0.9\textwidth]{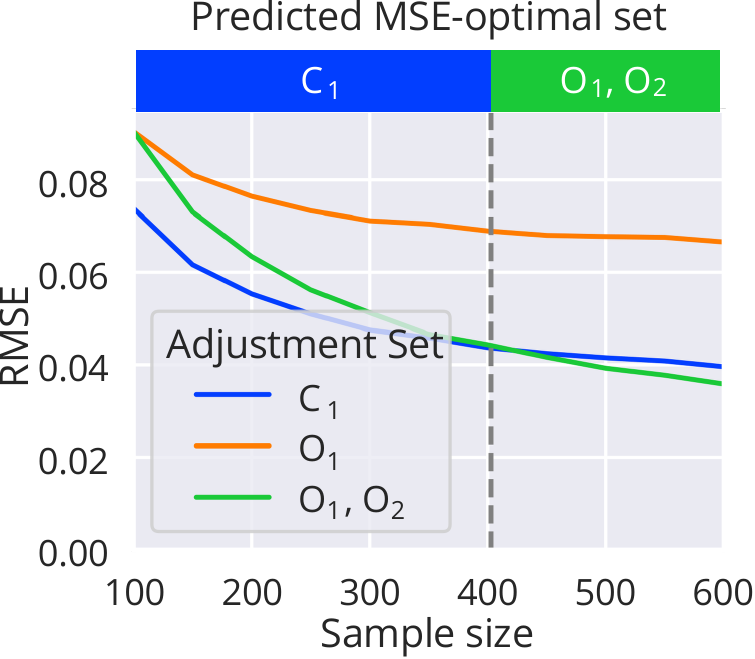}
        \caption{RMSE of $\hat \tau$ in $\mathcal{M}_2$}
    \end{subfigure}
 
    \caption{Two toy examples of causal models $\mathcal{M}_1$ and $\mathcal{M}_2$ and the root-mean squared error (RMSE) of the ordinary least squares estimator $\hat \tau$ of the causal effect $\tau$ of $A$ on $Y$ in $\mathcal{M}_1$ and $\mathcal{M}_2$, using different adjustment sets (10000 random seeds per set and sample size). The variables in $\mathcal{M}_1$ and $\mathcal{M}_2$ follow linear equations with the shown coefficients and additive Gaussian noise with a variance of 1 and a mean of 0. In both examples, $O=\{O_1, O_2\}$ is the asymptotically optimal adjustment set \citep{henckelGraphicalCriteriaEfficient2022a}. Depending on the sample size, a different adjustment set than $O$ gives the lowest root-mean squared error, as shown in plots (b) and (d). We show the predicted MSE-optimal set based on our sample size criterion on the top of the plot, while the dashed line shows the sample size for which we predict a change in an adjustment set outperforming another.} 
    \label{fig:examples}
\end{figure}
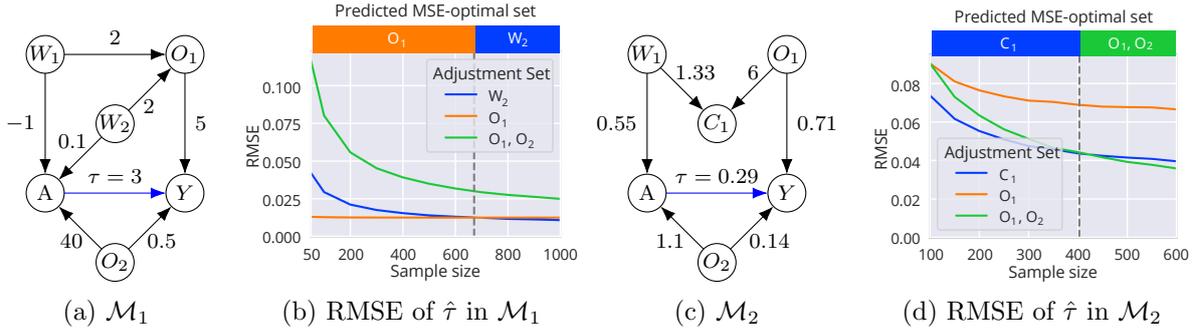

Variable selection is of the utmost importance for trustworthy causal inference \citep{brookhart2006variable, pearl2009causality, steiner2010importance}. Causal graphical models provide a powerful framework for understanding the dependencies among variables. These models help identify valid adjustment sets that yield unbiased estimates of causal effects \citep{Shipster2010OnTheValidity, perkovic2015completeCriterion}.

So far, methods based on causal graphs focus on \emph{valid} adjustment sets \citep{rotnitzkyEfficientAdjustmentSets2019, henckelGraphicalCriteriaEfficient2022a}, aiming for unbiasedness and asymptotic efficiency of the causal effect estimator. However, in finite sample size settings, the variance may dominate the bias, such that an \emph{invalid} adjustment set that does not satisfy the criteria for unbiasedness might be more suitable for estimation. By providing criteria that also consider invalid adjustment sets, we allow for extra flexibility in the choice of covariates. This is particularly practical when certain covariates are expensive to measure. Selecting the adjustment set with the smallest estimated mean squared error from a specified set of candidate adjustment sets allows us to omit these difficult-to-measure variables, offering a practical alternative that still provides accurate estimates.

Figure \ref{fig:examples} shows two examples in which invalid adjustment sets outperform the unbiased and asymptotically optimal adjustment set $O$ \citep{henckelGraphicalCriteriaEfficient2022a, rotnitzkyEfficientAdjustmentSets2019} in terms of mean squared error in finite sample cases. In the examples, tolerating a certain amount of omitted variable bias \citep{greene2003econometric, chernozhukov2022long, cinelli2019making} brings a substantial improvement in terms of variance. The interplay between bias and variance in finite samples is also influenced by phenomena such as bias unmasking and bias amplification \citep{Middleton_Scott_Diakow_Hill_2016, pearl2010onAClass, myers2011effects, bhattacharya2007do, wooldridge2016should}, where controlling for additional covariates can increase bias, either by revealing hidden biases or amplifying existing biases, respectively. 

In this paper, we describe how to find the adjustment set that optimizes the mean squared error in linear Gaussian causal models. Unlike previous work, we focus on finite sample properties of the estimator instead of its asymptotic behaviour. As a result, we demonstrate that, in certain settings, deliberately choosing an invalid adjustment set can be beneficial. We derive a sample size criterion which describes the conditions under which this is the case, assuming that the causal effect is estimated with the ordinary least squares estimator. Additionally, we develop graphical criteria to reduce the search space of the MSE-optimal adjustment set, reducing the additional computational effort required to identify it. In experiments on synthetic data, we show that this additional computational effort can be worthwhile in finite samples. Specifically, our method for covariate selection, based on these theoretical findings, matches or exceeds the performance of the asymptotically optimal adjustment set from \cite{henckelGraphicalCriteriaEfficient2022a} in linear Gaussian settings when the causal effect is estimated with ordinary least squares.

\section{Preliminaries}

\subsection{Linear Gaussian causal models}
We consider treatment effect estimation with causal graphical models, specifically with directed acyclic graphs. A causal directed acyclic graph $\mathcal{G}=(\mathcal{V},\mathcal{E})$ consists of a set of nodes $\mathcal{V}$ and a set of edges $\mathcal{E}$, where each node represents a random variable. A directed edge $V_i \rightarrow V_j$ for $i \neq j \in \{1, \dots, d\}$ between two nodes represents the direct causal effect of $V_i$ on $V_j$, and we say that $V_i$ is a \emph{parent} of $V_j$. We denote the set of parents of the variable $V$ in the graph $\mathcal{G}$ by $\mathrm{Pa}(V, \mathcal{G})$.

A sequence of nodes $\pi=(V_1, \dots, V_j)$ forms a \emph{path} if there exists an edge between all consecutive nodes in the sequence. If all edges point in the same direction, i.e. $V_i \rightarrow V_{i+1}$ for all $V_i \in \pi$, the path $\pi$ is a \emph{directed} path from $V_i$ to $V_j$, and $V_j$ is a \emph{descendant} of $V_i$. We use $\mathrm{De}(V_i, \mathcal{G})$ to denote the set of descendants of $V_i$ in $\mathcal{G}$, where we do not consider $V_i$ as a descendant of itself. A node $V$ is a \emph{collider} on a path $\pi$ if $\pi$ contains the structure $U \rightarrow V \leftarrow W$.

We use $X \indep_{\mathcal{G}} Y \mid Z$ to denote d-separation of two nodes $X$ and $Y$ given a set of nodes $Z$ in the graph $\mathcal{G}$. A definition of d-separation is given in Appendix~1 for convenience. Two nodes $X$ and $Y$ are d-connected given $Z$ if they are not d-separated given $Z$, which we denote by $X \not\indep_{\mathcal{G}} Y \mid Z$. Assuming the causal Markov and faithfulness assumptions, the d-separation $X \indep_{\mathcal{G}} Y \mid Z$ in the graph $\mathcal{G}$ corresponds to a conditional independence $X \indep
Y \mid~Z$ in the probability distribution of the corresponding random variables. 
Under these assumptions, the joint probability distribution of the random variables is \emph{Markov} to the graph $\mathcal{G}$, and factorizes as  $
    \mathrm{pr}(V_1, \dots , V_d) = \prod_{i=1}^d \mathrm{pr}\left\{V_i | \text{Pa}(V_i, \mathcal{G})\right\}$, where $\mathrm{pr}\left\{V_i | \text{Pa}(V_i, \mathcal{G})\right\}$ is the conditional probability of $V_i$ given its parents in $\mathcal{G}$. We assume that the following linear Gaussian causal model $\mathcal{M}$ holds: 
\begin{align}
    V_i &= \sum_{V_j \in \text{Pa}(V_i, \mathcal{G})} \beta_{ij} V_j + \epsilon_i, & \epsilon_i \sim \mathcal{N}(0, \sigma_i^2)  \qquad (i =1, \dots, d),
    \label{eq:causal-model}
\end{align}
where the noise terms $\epsilon_1, \dots, \epsilon_d$ are jointly independent.
Under a do-intervention on variable $V_i$ with value $x$, denoted by $do(V_i=x)$, the equation of $V_i$ is replaced with $x$. 

\subsection{Estimating the average treatment effect}
We aim to estimate the average treatment effect of a treatment variable $A$ on an outcome variable $Y$, assuming the following definition of the average treatment effect using the do-operator.
\begin{definition}[Average Treatment Effect $\tau$]
Let \( Y \) be the outcome variable and \( A \) be the treatment variable. The average treatment effect $\tau$ of the treatment \( A \) on the outcome \( Y \) is:
\[
\tau := \frac{\partial}{\partial a} E\{Y \mid \text{do}(A = a)\}. 
\] 
\end{definition}
We assume that all variables in the graph, except $A$ and $Y$, are pre-treatment variables. 
\begin{assumption}[Pre-treatment variables]
    No variables $V_i \in \mathcal{V}\setminus\{A,Y\}$ are descendants of the treatment $A$:
    \begin{equation*}
        \left( \mathcal{V} \setminus \{A,Y\} \right) \cap \mathrm{De}(A, \mathcal{G}) = \emptyset.
        \label{assumption:pretreatment}
    \end{equation*}
\end{assumption}
The pre-treatment assumption implies that we do not have any mediators $M$ in the graph $\mathcal{G}$, i.e. $M(\mathcal{G}) = \emptyset$. We define mediators as variables that block a directed path between the treatment $A$ and the outcome $Y$.
Assuming that all covariates are pre-treatment variables, the average treatment effect equals the coefficient $\beta_{ij}$ with $V_i = Y$ and $V_j = A$ for the linear model \eqref{eq:causal-model}. 

\subsection{Asymptotic optimality of adjustment sets}

Adjustment sets are used to estimate the causal effect of a variable, here the treatment $A$, on another, here the outcome $Y$, from observational data via covariate adjustment. We denote an average treatment effect estimator that uses covariate adjustment with the adjustment set $K$ by $\hat \tau_{K}$. A set of covariates $K$ is a \emph{valid} adjustment set if the estimator $\hat \tau_{K}$ returns an unbiased estimate of the true causal effect $\tau$ under correct model specification, for all probability distributions that are Markov to the causal graph $\mathcal{G}$. Whether an adjustment set is valid can be determined from the causal graph alone, e.g. with the sufficient back-door criterion \citep{pearl1993comment}, or a necessary and sufficient criterion developed by \cite{Shipster2010OnTheValidity} and \cite{perkovi2018complete}.

The \emph{optimal} adjustment set $O$ is the adjustment set with minimal asymptotic variance among all valid adjustment sets. It was first defined for ordinary least squares estimation in linear causal graphical models \citep{henckelGraphicalCriteriaEfficient2022a} and later extended to non-parametric models \citep{rotnitzkyEfficientAdjustmentSets2019}. We follow \cite{guoVariableEliminationGraph2023} for an intuitive definition of $O$.

\begin{definition} [Optimal adjustment set] 
    \citep{henckelGraphicalCriteriaEfficient2022a, guoVariableEliminationGraph2023} Let $\mathcal{G}=(\mathcal{V},\mathcal{E})$ be a causal directed acyclic graph with $A, Y \in \mathcal{V}$. For estimating the treatment effect $\tau$, the optimal adjustment set $O$ consists of the parents of mediators $M(\mathcal{G})$ that are not themselves mediators or the treatment, where mediators are defined to also include the outcome:
\begin{align*}
    O(\mathcal{G})  &\equiv \rm{Pa}\left\{ M(\mathcal{G}), \mathcal{G}) \right\} \setminus \left\{ M(\mathcal{G}) \cup \{A\} \right\}. 
\end{align*}
\end{definition}
Our setting is similar to the setting in \cite{henckelGraphicalCriteriaEfficient2022a} but with the additional assumption of Gaussianity. This enables us to consider all possible adjustment sets instead of only valid ones. Assuming a linear Gaussian causal model and ordinary least squares estimation, the asymptotic variance provided by any adjustment set $K$ is $\mathrm{aVar}(\hat \tau_{K}) = \sigma_{yy.ak} / \sigma_{aa.k}$, where $\sigma_{yy.x}$ denotes the conditional covariance of the variable $Y$ with itself, given the set of variables $X$, i.e. $\sigma_{yy.x} = \mathrm{var}(Y) - \mathrm{cov}(Y,X) \mathrm{cov}(X,X)^{-1}\mathrm{cov}(X,Y)$ \citep[prop. 1]{henckelGraphicalCriteriaEfficient2022a}. 
\cite{henckelGraphicalCriteriaEfficient2022a} show in Theorem 1 that the optimal adjustment set $O$ is asymptotically optimal in the sense that it provides an asymptotic variance $\mathrm{aVar}(\hat \tau_{O})$ that is smaller than or equal to the asymptotic variance provided by any other valid adjustment set $Z$. In this paper, we instead describe how to find an adjustment set that is not only asymptotically optimal, but finite sample optimal in terms of the mean squared error.

\section{Finding the MSE-optimal adjustment set}
\label{section:finding-the-set}
\subsection{Mean squared error optimality}
We aim to find the adjustment set that gives the most accurate average treatment effect estimator in terms of mean squared error for a given causal model $\mathcal{M}$ and sample size $n$. We call this set the \emph{MSE-optimal adjustment set}.

\begin{definition}[MSE-optimal adjustment set] Let $\tau$ be the average treatment effect in a ground truth causal model $\mathcal{M}$ with random variables $\mathcal{V}$. We define the MSE-optimal adjustment set $O_n(\mathcal{M}, \hat \tau_{K})$ as an adjustment set $K$ that minimizes the mean squared error of a given causal effect estimator $\hat \tau_{K}$ using $n$ datapoints $v_1, \dots v_n$ of the variables $V \in \mathcal{V}$, sampled from the observational distribution corresponding to the model $\mathcal{M}$:
    \begin{equation}
        O_n(\mathcal{M}, \hat \tau_{K}) =\underset{K \subseteq \mathcal{V} \setminus \{A,Y\}}{\mathrm{argmin}}E_{v_1,  \dots, v_n \sim \mathcal{M}}\{ ( \hat \tau_K - \tau )^2\}
    \end{equation}
\end{definition}

We focus on the MSE-optimal adjustment set in the setting where $\mathcal{M}$ is linear Gaussian and $\hat \tau$ is the ordinary least squares estimator. For simplicity, we may omit $\mathcal{M}$ and $\hat\tau$ to ease notation and denote the MSE-optimal adjustment set as $O_n$. 

In many cases, including our own experiments (see Figure~\ref{fig:examples}), $O_n$ converges to $O$ as the sample size approaches infinity. However, in some cases, it is possible that $O_n$ differs from $O$ asymptotically. For example, consider a causal model $\mathcal{M}$, where the outcome is $Y = A + O_1 + O_2 + \epsilon_Y$ and the treatment is $A = 2O_1 - 2O_2 + \epsilon_A$ with $\epsilon_A, \epsilon_Y, O_1, O_2 \sim \mathcal{N}(0, 1)$. 
Here, the simple model with an empty adjustment set $K = \emptyset$ has zero bias $B(\hat \tau_{\emptyset}) = \text{cov}(A, Y) / \text{var}(A) - 1 = 0$. 
It has an asymptotic variance of $\text{aVar}(\hat \tau_{\emptyset}) = \text{var}(Y \mid A) / \text{var}(A) = 1/3$, which is lower than the asymptotic variance of the optimal adjustment set $\text{aVar}(\hat \tau_{O_1 \cup O_2}) = \text{var}(Y \mid A , O_1 ,O_2) / \text{var}(A \mid O_1,O_2) = 1$. In this example, the MSE-optimal adjustment set $O_n$ is also asymptotically the empty set $\emptyset$ and does not converge to $O$.

\subsection{Sample size criterion}
As demonstrated in Figure \ref{fig:examples}, the adjustment set that yields the lowest mean squared error for predicting the average causal effect can depend on the sample size. We present a criterion to compare two adjustment sets for treatment effect estimation given a linear Gaussian model $\mathcal{M}$ and sample size $n$. Two adjustment sets can be compared based on their set sizes, and the estimator's asymptotic variances and biases, for a given sample size as follows.

\begin{theorem}[Sample Size Criterion]
\label{sample-size-criterion}
Let $K$ and $L$ be two adjustment sets for estimating the causal effect $\tau$ with the ordinary least squares estimator, denoted as $\hat \tau_{K}$ or $\hat \tau_{L}$ respectively. We assume $|K| < n - 3$ and $|L| < n - 3$. If the squared bias $B^2(\hat{\tau}_{K})$ is larger than the squared bias $B^2(\hat{\tau}_{L})$, then the following condition for the sample size, denoted by $n$, is necessary and sufficient to ensure a lower expected mean squared error of $\hat{\tau}_{K}$ compared to $\hat{\tau}_{L}$:
\begin{equation}
\label{eq:sample-size-criterion}
    n < \frac{\mathrm{aVar}(\hat \tau_{L}) - \left(\frac{n - |L| - 3}{n - |K| - 3}\right)\mathrm{aVar}(\hat \tau_{K})}{B^2(\hat{\tau}_{K}) - B^2(\hat{\tau}_{L})} + |L| + 3. 
\end{equation}
\end{theorem}
We present the proof in Appendix~2. Intuitively, the bias advantage of adjustment set $L$ is scaled by $n$ and then roughly compared to its disadvantage in asymptotic variance. Since good variance properties might outweigh a given bias, considering invalid adjustment sets for treatment effect estimation becomes important when $n$ is finite.

\subsection{Graphical criteria}

With Theorem~\ref{sample-size-criterion}, we have introduced a criterion to compare adjustment sets in terms of their mean squared error. However, it does not provide us with an efficient way to search for MSE-optimal adjustment set. A straightforward option is to search over the power set of all covariates, which is inefficient and scales poorly with the number of covariates. In the following, we will show that the search space for $O_n$ can be limited to a smaller space than the power set of all covariates. For linear Gaussian causal models, some variables or variable combinations can be excluded from the adjustment set, solely based on the graph $\mathcal{G}$.

For example, variables that are d-separated from $Y$ given any $K \subseteq \mathcal{V} \setminus \{A, Y\}$ in $\mathcal{G}'=\mathcal{G}\setminus (A \rightarrow Y)$ always increase mean squared error, as we show in Lemma~B10 in the Supplementary Material. This includes instrumental variables which only affect the outcome through the treatment. Specifically, adding an instrumental variable $I$ to an adjustment set $K$ might improve precision of the estimator yielded by $K \cup I$ compared to $K \setminus I$, if $K$ is an invalid adjustment set, as $I$ can reduce $\sigma_{yy.aki}$ compared to $\sigma_{yy.ak}$ via open confounding paths. However, this is always outweighed by the amount of bias amplification added by conditioning on $I$.

Additionally, certain precision variables \citep{brookhart2006variable} and confounding variables are never necessary to achieve an optimal mean squared error, as we will explain in the following. We use the following definition of precision variables:

\begin{definition}[Precision Variables]
\label{def:precision}
    Let $\mathcal{V}$ be the set of variables in a directed acyclic graph $\mathcal{G}=(\mathcal{V}, \mathcal{E})$ describing the causal relations of $\mathcal{V}$, with $A, Y \in \mathcal{V}$. Let $\mathcal{G}'$ be the graph obtained from $\mathcal{G}$ by removing the edge $A \rightarrow Y$. For estimating the causal effect $\tau$, a variable $V_i \in \mathcal{V} \setminus \{A, Y \}$ is a precision variable, if in  $\mathcal{G}'$ it is d-separated from $A$ given $K$ for all $K \subseteq \mathcal{V} \setminus \{A, Y \}$, and d-connected to $Y$ given $L$,  for some $L \subseteq \mathcal{V} \setminus \{A, Y \}$. We denote the set of precision variables in $\mathcal{G}$ by $\mathcal{P}$.
\end{definition}

Generally, precision variables can increase the precision of the treatment effect estimate \allowbreak\citep{brookhart2006variable}.  However, if the set of covariates is relatively large compared to the sample size, a precision variable that only contains little information about the outcome may increase the variance of the ordinary least squares estimator. In Appendix~2, we provide a reformulation of the sample size criterion that shows when adding a set of precision variables improves mean squared error. Based on the causal graph alone, we can exclude the following precision variables when searching for the MSE-optimal adjustment set.

\begin{definition}[Suboptimal precision variables]
\label{def:suboptimal-precision}
  Let $P \in \mathcal{P}$ be a precision variable in the set of variables $\mathcal{V}$ in a directed acyclic graph $\mathcal{G}=(\mathcal{V}, \mathcal{E})$.  If there exists another precision variable $P^* \in \mathcal{P}$, such that all paths from $P$ to $Y$ in $\mathcal{G}'=\mathcal{G}\setminus (A \rightarrow Y)$ are blocked given $P^*$ and any other set $Z \subseteq \mathcal{V} \setminus \{A, Y \}$, then $P$ is a suboptimal precision variable. We call $\mathcal{S}^{P}$ the set of all suboptimal precision variables in $\mathcal{G}$.
\end{definition}

\begin{wrapfigure}{r}{0.35\textwidth} 
    \centering
    \vspace{-10pt}
    \begin{tikzpicture}[
        font=\fontsize{8}{10}\selectfont,
        node distance=1.3cm,
        on grid,
        auto,
        block/.style={circle, draw, inner sep=0pt, outer sep=0pt, minimum size=0.5cm}
    ]
    
    \node [block] (A) {A};
    \node [block, right=of A] (Y) {$Y$};
    \node [block, above=of Y] (O1) {$O_1$};
    \node [block, above=of A] (I1) {$S_1$};
    \node [block, right=of Y] (P1) {$O_2$};
    \node [block, above=of P1] (I2) {$S_2$};
    \node [block, right=of P1] (I3) {$S_3$};
    \node [block, below=of P1] (P) {$P_1$};
    \node [block, left=of P] (O3) {$O_3$};
    \node [block, right=of P] (O4) {$O_4$};
    \node [block, left=of O3] (IV1) {$I_1$};
    
    \draw[-{Latex[length=2mm]}, blue] (A) -- (Y) node[midway, above, black] {$\tau$};
    \draw[-{Latex[length=2mm]}] (I1) -- (O1) node[midway, right, pos=0.1] {};
    \draw[-{Latex[length=2mm]}] (O1) -- (Y) node[midway, right] {};
    \draw[-{Latex[length=2mm]}] (I1) -- (A) node[midway, left] {};    
    \draw[-{Latex[length=2mm]}] (P1) -- (Y) node[midway, left] {};   
    \draw[-{Latex[length=2mm]}] (I2) -- (P1) node[midway, left] {};    
    \draw[-{Latex[length=2mm]}] (I3) -- (P1) node[midway, left] {};  
    \draw[-{Latex[length=2mm]}] (P) -- (O3) node[midway, left] {};  
    \draw[-{Latex[length=2mm]}] (O3) -- (Y) node[midway, left] {};  
    \draw[-{Latex[length=2mm]}] (P) -- (O4) node[midway, left] {}; 
    \draw[-{Latex[length=2mm]}] (O4) -- (Y) node[midway, left] {}; 
    \draw[-{Latex[length=2mm]}] (IV1) -- (A) node[midway, left] {}; 
    \end{tikzpicture}
    \caption{Example graph $\mathcal{G}_3$, $\mathcal{G}'_3$ is the same without the edge between $A$ and $Y$.}
    \label{fig:g3}
    \vspace{-10pt}
\end{wrapfigure}
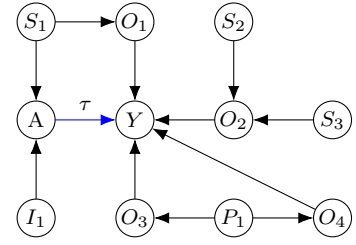

For an example, see $\mathcal{G}_3$ in Figure~\ref{fig:g3}, with the precision variables $\mathcal{P}=\{O_2, O_3, O_4, S_2, S_3, P_1 \}$. The variables $S_2$ and $S_3$ are suboptimal precision variables, as the precision variable $P^*=O_2$ blocks the paths $( S_2, O_2, Y )$ and $( S_3, O_2, Y )$. Since there is no other single precision variable that blocks all paths from $P_1$ to $Y$, $P_1$ is not a suboptimal precision variable. 

Certain variables that are related to both the treatment and outcome can also be excluded from the MSE-optimal adjustment set. Commonly, pre-treatment variables that are related to both the outcome and the treatment, are referred to as \emph{confounders}. In the context of directed acyclic graphs, confounders are usually defined as variables $V$ that are a \emph{common cause} of $A$ and $Y$ \citep{pearl2009causality}. Here, we propose an extended definition of confounders that also includes pre-treatment colliders and other pre-treatment variables that are d-connected to confounders, because we can then define a similar criterion for suboptimality in terms of the mean squared error as for the precision variables.

\begin{definition}[Extended confounding variables]
\label{def:confounding}
    Let $\mathcal{V}$ be the set of variables in a directed acyclic graph $\mathcal{G}=(\mathcal{V}, \mathcal{E})$ describing the causal relations of $\mathcal{V}$, with $A, Y \in \mathcal{V}$.  Let $\mathcal{G}'$ be the graph obtained from $\mathcal{G}$ by removing the edge $A \rightarrow Y$. For estimating the causal effect $\tau$, a variable $V_i \in \mathcal{V}  \setminus \{A, Y\}$ is in the set of extended confounding variables if it is d-connected to $A$ in $\mathcal{G}'$, given some $K \subseteq \mathcal{V} \setminus \{A, Y\}$, and d-connected to $Y$ in $\mathcal{G}'$, given some $L \subseteq \mathcal{V} \setminus \{A, Y\}$. We denote the set of extended confounding variables in $\mathcal{G}$ by $\mathcal{W}$.
\end{definition}
The following extended confounding variables can be considered suboptimal in terms of mean squared error.

\begin{definition}[Suboptimal confounding variables] 
\label{def:suboptimal-confounding}
   Let $W \in \mathcal{W}$ be an extended confounding variable in the set of variables $\mathcal{V}$ in a directed acyclic graph $\mathcal{G}=(\mathcal{V}, \mathcal{E})$. Let $\mathcal{G}'$ be the graph obtained from $\mathcal{G}$ by removing the edge $A \rightarrow Y$. If there exists another extended confounding variable $W^* \in \mathcal{W}$, such that $W \indep_{\mathcal{G}'} Y \mid W^* \cup Z$ and $W^* \indep_{\mathcal{G}'} A \mid W \cup Z$ for any set $Z \subseteq \mathcal{V} \setminus \{A, Y\}$, then $W$ is a suboptimal confounding variable. We call $\mathcal{S}^{W}$ the set of all suboptimal confounding variables in $\mathcal{G}$.
\end{definition}

Again, see $\mathcal{G}_3$ in Figure~\ref{fig:g3} for an example, with the extended confounding variables $\mathcal{W}=\{S_1, O_1 \}$. The variable $S_1$ is a suboptimal confounding variable, since all paths between $S_1$ and $Y$ are blocked in $\mathcal{G}'_3$ given $O_1$ and any other set, i.e. the path $(S_1, O_1, Y)$, and since the only path between $O_1$ and $A$, i.e., $(O_1, S_1, A)$, is blocked by $S_1$. For another more complex example, consider the graph of $\mathcal{M}_1$ in Figure~\ref{fig:examples}. Here, all paths between $W_1$ and $Y$ in $\mathcal{G}'_2$ are blocked given $O_1$ and any other set. However, not all paths between $O_1$ and $A$ in $\mathcal{G}'_2$ are blocked given $W_1$ and any other set. Specifically, 
the path $(O_1, W_2, A)$ is open given $W_1$. Therefore, the extended confounding variable $W_1$ is not a suboptimal confounding variable.

\begin{theorem}[MSE-optimal adjustment set candidates] \label{theorem:O_n} Let $\mathcal{O}_n(\mathcal{M}, \hat \tau_{K})$ be the set of all MSE-optimal adjustment sets for the causal linear Gaussian model $\mathcal{M}$ and sample size $n$. Let $\mathcal{V}$ be the variables in $\mathcal{G}$. There exists at least one MSE-optimal adjustment set $O^*_n \in \mathcal{O}_n(\mathcal{M}, \hat \tau)$, such that every variable $V_j \in \mathcal{V} \setminus \{A, Y\}$ in $O^*_n$ is either (i) a confounding variable $W \in \mathcal{W}$ that is not suboptimal or (ii) a precision variable $P \in \mathcal{P}$ that is not suboptimal:
\begin{equation}
    \label{eq:O_n}
    O^*_n \cap \mathcal{V}\setminus \{A, Y\} \subseteq  \mathcal{W} \setminus \mathcal{S}^W \cup \mathcal{P}  \setminus \mathcal{S}^P 
\end{equation}
\end{theorem}

Theorem~\ref{theorem:O_n} helps to reduce the search space by excluding certain single variables. We report d-separation properties of precision variables, extended confounding variables and the remaining variables in Lemma~B4, Lemma~B5 and Lemma~B6 in the Supplementary Material, which were used for the proof of Theorem~\ref{theorem:O_n} and may be of independent interest. The search space for the MSE-optimal adjustment set can be reduced even further by excluding certain variable combinations, which we call \emph{forbidden combinations}.

\begin{theorem}[Forbidden combinations]
\label{theorem:combinations}  
 Let $L$ be a set of pre-treatment variables, including the covariate $L_i$. Let $L_{-i}$ be the covariates in $L$ without $L_i$, i.e. $L_{-i}=L \setminus L_i$. If $L_i$ is d-separated from $Y$ given $L_{-i}$ and any other set $K \subseteq \mathcal{V} \setminus \{A, Y\}$ in $\mathcal{G}'=\mathcal{G} \setminus (A \rightarrow Y)$, then any adjustment set $X$ with $L \subseteq X$ can not be MSE-optimal for any $n$.
\end{theorem}
Finally, we can exclude certain valid adjustment sets from the search space.

\begin{theorem}[Suboptimal valid adjustment sets]
    \label{theorem:valid-sets}
    Let $\mathcal{O}_n(\mathcal{M}, \hat \tau_{K})$ be the set of all MSE-optimal adjustment sets for the causal linear Gaussian model $\mathcal{M}$ and sample size $n$. Let $L$ be a valid adjustment set for estimating $\tau$. If $L$ is not the optimal adjustment set $O$, i.e. $L \neq O$, and it has larger or equal size compared to $O$, i.e. $|L| \geq O$, then the mean squared error yielded by $L$ is larger than or equal to the mean squared error yielded by $O$, such that if $L \in \mathcal{O}_n(\mathcal{M}, \hat \tau_{K})$, then $O \in \mathcal{O}_n(\mathcal{M}, \hat \tau_{K})$, and we can exclude $L$ from the search space, as we already consider $O$.
\end{theorem}
With our graphical criteria, we can reduce the number of potential adjustment sets in the search space for $\mathcal{G}_3$ in Figure~\ref{fig:g3} from 512 to only 28 adjustment sets. In $\mathcal{M}_1$ from Figure~\ref{fig:examples}, we can reduce the number of adjustment sets from 16 to 9, and in $\mathcal{M}_2$ from Figure~\ref{fig:examples} from 16 to 5. For a detailed explanation, see Appendix~3. 

We use Algorithm~1 in Appendix~4 to find the MSE-optimal adjustment set $O_n$ and estimate the average treatment effect. The algorithm first decreases the size of the search space for adjustment sets with Theorem~\ref{theorem:O_n}, Theorem~\ref{theorem:combinations} and Theorem~\ref{theorem:valid-sets}, and then chooses the adjustment set with the smallest estimated mean squared error from the remaining sets. Our code is available on \href{https://github.com/nadjarutsch/InvalidAdjustmentSets}{github}.

\section{Experiments}
We estimate the mean squared error of the ordinary least squares treatment effect estimator $\hat \tau_{K}$ for each potential adjustment set $K$ to find the estimated MSE-optimal adjustment set $\hat O_n$ in the examples from Figure~\ref{fig:examples}. Then, we compare $\hat O_n$, the ground truth MSE-optimal adjustment set $O_n$ and the asymptotically optimal adjustment set $O$ in terms of the provided mean squared error. Table~\ref{tab:main-results} shows that, in the causal model $\mathcal{M}_1$, $\hat O_n$ outperforms $O$ in small sample sizes, and performs competitively in larger sample sizes. We get qualitatively similar results for the causal model $\mathcal{M}_2$ from Figure~\ref{fig:examples} (right), which are shown in Appendix~4.

\begin{table}[htbp]
\centering
\footnotesize              
\setlength{\tabcolsep}{4pt}

\caption{Comparison of $O$ and $\hat O_n(\mathcal{M}_1, \hat \tau_{K})$ for $\mathcal{M}_1$
         from Figure~\ref{fig:examples} (left), 10\,000 random seeds.}
\label{tab:main-results}

\begin{tabular}{@{}lcccc@{}}
\toprule
Sample size &
\multicolumn{1}{c}{MSE for $O$ (Mean $\pm$ SD)} &
\multicolumn{1}{c}{MSE for $\hat O_n$ (Mean $\pm$ SD)} &
\multicolumn{1}{c}{MSE for $O_n$ (Mean $\pm$ SD)} &
$O_n$ \\
\midrule
10   & 0.1234 (0.2379) & \textbf{0.0926 (0.2343)} & 0.0003 (0.0003) & $O_1$ \\
20   & 0.0403 (0.0622) & \textbf{0.0306 (0.0673)} & 0.0002 (0.0002) & $O_1$ \\
30   & 0.0247 (0.0373) & \textbf{0.0182 (0.0374)} & 0.0002 (0.0001) & $O_1$ \\
40   & 0.0172 (0.0252) & \textbf{0.0135 (0.0270)} & 0.0002 (0.0001) & $O_1$ \\
50   & 0.0136 (0.0199) & \textbf{0.0103 (0.0201)} & 0.0002 (0.0001) & $O_1$ \\
100  & 0.0064 (0.0091) & \textbf{0.0048 (0.0095)} & 0.0002 (0.0001) & $O_1$ \\
500  & 0.0012 (0.0017) & \textbf{0.0010 (0.0018)} & 0.0002 (0.0000) & $O_1$ \\
1000 & 0.0006 (0.0009) & \textbf{0.0005 (0.0009)} & 0.0001 (0.0002) & $W_2$ \\
\bottomrule
\end{tabular}
\end{table}

\section{Discussion}

When the sample size is small compared to the number of variables, our current bias estimation method may be impractical. Especially when the smallest valid adjustment set $Z$ is larger than the sample size, i.e. $|Z| > n$, the ordinary least squares estimator is no longer feasible. For this setting, we propose an extension to our approach, where we select the adjustment set with the smallest estimated variance. We show the results of this extension in Appendix~4. Using the adjustment set with the smallest variance yields an even larger advantage over the optimal adjustment set $O$, but it comes with the disadvantage of bias-dominated estimates in larger sample size settings.

\section*{Acknowledgement}
N. Rutsch and S. L. van der Pas have received funding from the European Research Council (ERC) under the European Union’s Horizon Europe program under Grant agreement No. 101074082. Views and opinions expressed are however those of the author(s) only and do not necessarily reflect those of the European Union or the European Research Council Executive Agency. Neither the European Union nor the granting authority can be held responsible for them.

\section*{Supplementary Material}
The Supplementary Material includes definitions, proofs, an algorithm and a table with results from additional experiments.

\bibliographystyle{plainnat}
\bibliography{preprint}

\appendix
\label{SM}

\section*{Appendix 1}
\subsection*{Definitions}
\label{A:d-separation}

We first introduce the notion of collider on a path.
Let \( \mathcal{G} = (\mathcal{V}, \mathcal{E}) \) be a directed acyclic graph. A path $\pi$ between a variable $X \in \mathcal{V}$ and another variable $Y \in \mathcal{V}$ is a sequence of distinct nodes (X, \dots, Y) such that any two consecutive nodes in the sequence are adjacent in $\mathcal{G}$. A \emph{collider} on a path $\pi$ between $X$ and $Y$ is a node $W \in \mathcal{V} \setminus \{X, Y\}$ for which there are two incoming edges $\to W \gets$ on the path $\pi$. A \emph{non-collider}  on a path $\pi$ is instead defined as a node $W \in \mathcal{V} \setminus \{X, Y\}$ for which there are not two incoming edges on  the path $\pi$.
Given these definitions, we can now report the d-separation definition:

\begin{definition}[d-separation]\citep{pearl2009causality}
Let \( \mathcal{G} = (\mathcal{V}, \mathcal{E}) \) be a directed acyclic graph. Consider three disjoint subsets of nodes \( X, Y, \) and \( Z \) within \( \mathcal{V} \). The sets \( X \) and \( Y \) are d-separated given $Z$ in \( \mathcal{G} \) if and only if every path between any node in \( X \) and any node in \( Y \) is blocked by \( Z \). A path is  blocked if it contains a node $ W \in \mathcal{V} \setminus \{X, Y\}$ satisfying one of the following conditions:
\begin{enumerate}
    \item \( W \) is a non-collider on the path, and \( W \in Z \).
    \item \( W \) is a collider on the path, and neither \( W \) nor any of its descendants are in \( Z \).
\end{enumerate}
\end{definition}
In some proofs, we use the notion of $K$-irreducible adjustment sets, which we define as follows.
\begin{definition}[$K$-irreducible adjustment set]
\label{irreducible}
Let \( \mathcal{G} = (\mathcal{V}, \mathcal{E}) \) be a directed acyclic graph and let $Z \subseteq \mathcal{V} \setminus (A \rightarrow Y)$ with $K \subseteq Z$ be a valid adjustment set with respect to estimating the average treatment effect $\tau$ of variable $A$ on variable $Y$. The adjustment set $Z$ is $K$-irreducible if there exists no subset $Z' \subseteq Z$ with $K \subseteq Z'$, such that $Z'$ is a valid adjustment set.
\end{definition}
This definition is equivalent to the notion of $M$-minimality used in \cite{VANDERZANDER20191}. To avoid confusion with minimum size adjustment sets, we use the term irreducibility instead.

Furthermore, we use the following properties of $K$-irreducible adjustment sets.
\begin{proposition}[Properties of $K$-irreducible adjustment sets]
\label{prop:irreducible}
Let \( \mathcal{G} = (\mathcal{V}, \mathcal{E}) \) be a directed acyclic graph and let $K \subseteq \mathcal{V} \setminus (A \rightarrow Y)$. Then, there exists $Z \subseteq \mathcal{V} \setminus\{A,Y\}$, such that $K \cup Z$ is a $K$-irreducible adjustment set with respect to estimating the average treatment effect $\tau$ of variable $A$ on variable $Y$. Also, any variable $Z_i \in Z$ is an extended confounding variable. 
\end{proposition}
\begin{proof}
    The first point holds since $\mathcal{V} \setminus \{A, Y\}$ is a valid adjustment set under the pre-treatment assumption. We show the second point by contradiction. Assume that there exists a variable $Z_i \in Z$ that is not an extended confounding variable. Then, $Z_i$ is either d-separated from $A$ or from $Y$ in $\mathcal{G}'$ given any set $X \subseteq\{A, Y\}$. It follows that $K \cup Z \setminus Z_i$ is a valid adjustment set, which contradicts that $K \cup Z$ is $K$-irreducible.
\end{proof}

\section*{Appendix 2}

In the Appendix, we use similar notation as in \cite{henckelGraphicalCriteriaEfficient2022a}. We denote the covariance matrix of a variable set $X$ with $\Sigma_{xx}$, and the covariance matrix between $X$ and $Z$ with $\Sigma_{xz}$. When $|X| =1$ or $|Y|=1$, we write $\sigma_{xz}$ to denote a vector or scalar, similarly we write $\sigma_{xx}$ when $|X| =1$. Furthermore, we use $\Sigma_{xx.z} = \Sigma_{xx} - \Sigma_{xz}\Sigma^{-1}_{zz}\Sigma_{zx}$ or $\sigma_{xx.z}$ respectively. We use $\beta_{yx.s}$ to denote the coefficient $\beta_{yx.s}$ of $X$ on $Y$ when regressing $Y$ on $X$ and $S$.
We use the following results from \cite{henckelGraphicalCriteriaEfficient2022a}, where $X_{-i}=X \setminus X_i$.

\begin{lemma} \citep[Lemma C.2]{henckelGraphicalCriteriaEfficient2022a}
\label{lemma:ci-variance}
  Let $(X^T, Y^T, T^T, S^T, W^T)^T$, with $T,\,S,$ and $W$ possibly of length zero, be a mean $0$ random vector with finite variance, such that $X = (X_1,\dots,X_{k_x})^T$ and $ Y = (Y_1,\dots,Y_{k_y})^T.$ If $T \indep Y \mid W, S, X$ and $S \indep X \mid W, T$, then
  \begin{enumerate}[(a)]
    \item $\sigma_{x_ix_i.x_{-i}wt} \leq \sigma_{x_ix_i.x_{-i}ws}$, 
    \item $\sigma_{y_jy_j.xws} \leq \sigma_{y_jy_j.xwt}$,
  \end{enumerate}
  for all $i \in \{1,\dots,k_x\}$ and $j \in \{1,\dots,k_y\}$.
\end{lemma}
We also use the following Corollary derived from Lemma~\ref{lemma:ci-variance}:
\begin{corollary}
\label{cor:ci-variance-separated}
  Let $(X^T, Y^T, T^T, S^T, W^T)^T$, with $T,\,S,$ and $W$ possibly of length zero, be a mean $0$ random vector with finite variance, such that $X = (X_1,\dots,X_{k_x})^T$ and $ Y = (Y_1,\dots,Y_{k_y})^T.$ If $S \indep X \mid W, T$, then
  \begin{enumerate}[(a)]
    \item $\sigma_{x_ix_i.x_{-i}wt} \leq \sigma_{x_ix_i.x_{-i}ws}$,
  \end{enumerate}
  for all $i \in \{1,\dots,k_x\}$ and $j \in \{1,\dots,k_y\}$. If $T \indep Y \mid W, S, X$, then
  \begin{enumerate}[(b)]
    \item $\sigma_{y_jy_j.xws} \leq \sigma_{y_jy_j.xwt}$,
  \end{enumerate}
  for all $i \in \{1,\dots,k_x\}$ and $j \in \{1,\dots,k_y\}$.
\end{corollary}
\begin{proof}
    Follows directly from the proof in \citet[Lemma C.2]{henckelGraphicalCriteriaEfficient2022a}, where each property is proved independently based on the corresponding d-separation.
\end{proof}

\begin{lemma} \citep[Lemma C.5]{henckelGraphicalCriteriaEfficient2022a}
\label{lemma:equal-coeffs}
Let \((X^T, Y^T, S^T, T^T)^T\), with \(S\) possibly of length zero, be a mean \(0\) random vector with finite variance. If \(T \indep X \mid S\) or \(T \indep Y \mid X, S\), then \(\beta_{yx.s} = \beta_{yx.st}\). Furthermore, if \(T \indep Y \mid X, S\), then $\Sigma_{yy.xst} = \Sigma_{yy.xs}$.
\end{lemma}
Furthermore, we use the following results from \citet{pmlr-v213-pena23a}:

\begin{lemma} \citep[Lemma~16]{pmlr-v213-pena23a}
        \label{lemma:equal-cov}
        Consider a path diagram $\mathcal{G}$. Let $X$, $Y$ and $W$ be nodes and $Z$ a set of nodes. If $X \indep_{\mathcal{G}} W \mid Z$ or $Y \indep_{\mathcal{G}} W \mid Z$, then $\sigma_{xy.zw} = \sigma_{xy.z}$. 
    \end{lemma}

\subsection*{Proof of Theorem 1 (Sample size criterion)}
\label{app:proof-sample-size-criterion}
The variance of the ordinary least squares estimator with any adjustment set \(K\) is
\begin{align}
    \label{eq:ols-var-1}
    \mathrm{var}(\hat \tau_{K}) \;=\; 
    E\!\Bigl(\frac{\sigma_{yy.ak}}{\mathrm{RSS}_{a.k}}\Bigr),
\end{align}
where \(\mathrm{RSS}_{a.k}\) is the residual sum of squares after regressing \(A\) on \(K\). Since \(\mathrm{RSS}_{a.k} / \sigma_{aa.k}\) follows an inverse-chi-squared distribution, and recalling \(\mathrm{aVar}(\hat \tau_{K}) = \sigma_{yy.ak} \,/\, \sigma_{aa.k}\), it follows that
\begin{align}
    \label{eq:ols-var-2}
    \mathrm{var}(\hat \tau_{K}) 
    \;=\; 
    \frac{\mathrm{aVar}(\hat \tau_{K})}{\,n - |K| - 3\,}.
\end{align}
Hence,
\begin{align*}
    \mathrm{MSE}(\hat \tau_{K}) - \mathrm{MSE}(\hat \tau_{L})
    =
    B^2(\hat \tau_{K}) - B^2(\hat \tau_{L})
    +
    \frac{\mathrm{aVar}(\hat \tau_{K})}{n - |K| - 3} 
    -
    \frac{\mathrm{aVar}(\hat \tau_{L})}{n - |L| - 3}.
\end{align*}
It follows that \(\mathrm{MSE}(\hat \tau_{K}) < \mathrm{MSE}(\hat \tau_{L})\) if and only if
\begin{align}
    B^2(\hat \tau_{K}) - B^2(\hat \tau_{L})
    <
    \frac{\mathrm{aVar}(\hat \tau_{L})}{n - |L| - 3}
    -
    \frac{\mathrm{aVar}(\hat \tau_{K})}{n - |K| - 3} ,
    \label{eq:ssc-rewritten}
\end{align}
which can be rewritten into Equation~(3) from the main paper.

\subsection*{Partitioning of covariates}
For a complete partitioning of all covariates $\mathcal{V} \setminus \{A, Y\}$, we define the set of \emph{irrelevant variables} $\mathcal{I}$, such that, with $\mathcal{P}$ and $\mathcal{W}$ as in Definition~4 and Definition~5 of the main paper,
\begin{align*}
    \mathcal{V}\setminus \{A, Y\} &= \mathcal{P} \cup \mathcal{W} \cup \mathcal{I}.  
\end{align*}

\begin{definition}[Irrelevant Variables]
\label{def:irrelevant}
Let $\mathcal{V}$ be the set of variables in a directed acyclic graph $\mathcal{G}=(\mathcal{V}, \mathcal{E})$ describing the causal relations of $\mathcal{V}$, with $A, Y \in \mathcal{V}$. For estimating the causal effect $\tau$, a variable $V_i \in \mathcal{V}\setminus\{A,Y\}$ is an irrelevant variable, if in  $\mathcal{G}'=\mathcal{G} \setminus (A \rightarrow Y)$ it is d-separated from $Y$ given $K$ for all $K \subseteq \mathcal{V} \setminus \{A,Y\}$. We denote the set of irrelevant variables in $\mathcal{G}$ by $\mathcal{I}$.
\end{definition}
To see that precision, extended confounding, and irrelevant variables form a complete partitioning of all variables, we recall that both precision and extended confounding variables $V_i \in \mathcal{P} \cup \mathcal{W}$ require that there exists a set $Z \subseteq \mathcal{V} \setminus \{A, Y\}$, such that $V_i \not\indep_{\mathcal{G}'} Y \mid Z$ in $\mathcal{G}'$. Since precision variables additionally require that (1) \emph{for all} $X \subseteq \mathcal{V} \setminus \{A, Y\}$, it holds that $V_i \indep_{\mathcal{G}'} A \mid X$, and extended confounding variables require that (2) there exists at least one set $X \subseteq \mathcal{V} \setminus \{A, Y\}$, such that $V_i \not\indep_{\mathcal{G}'} A \mid X$, it follows that $\mathcal{P}$ and $\mathcal{W}$ are disjunct. Furthermore, for any variable $V_i \in \mathcal{V}\setminus \{A,Y\}$, either (1) or (2) holds. It follows that the union $\mathcal{P} \cup \mathcal{W}$ is a complete partitioning of all variables, for which there exists a set $Z \subseteq \mathcal{V} \setminus \{A, Y\}$, such that $V_i \not\indep_{\mathcal{G}'} Y \mid Z$ in $\mathcal{G}'$.

The negation of $\mathcal{P} \cup \mathcal{W}$ in $\mathcal{V} \setminus \{A, Y\}$ consists of all variables $V_i \in \mathcal{V} \setminus \{A, Y\}$, for which there does not exist a set $Z \subseteq \mathcal{V} \setminus \{A, Y\}$, such that $V_i \not\indep_{\mathcal{G}'} Y \mid Z$ in $\mathcal{G}'$. This is equivalent to the set of all variables $V_i \in \mathcal{V} \setminus \{A, Y\}$, for which $V_i \indep_{\mathcal{G}'} Y \mid Z$ in $\mathcal{G}'$ for all $Z \subseteq \mathcal{V} \setminus \{A, Y\}$. It follows that the negation of $\mathcal{P} \cup \mathcal{W}$ in $\mathcal{V} \setminus \{A, Y\}$ is the set of irrelevant variables $\mathcal{I}$, and hence $\mathcal{P} \cup \mathcal{W} \cup \mathcal{I}$ is a complete partitioning of all variables $\mathcal{V}\setminus \{A, Y\}$.

\subsection*{d-separation properties}

\begin{lemma}[d-separation properties of precision variables]
\label{lemma-d-separation-P}
Let $P \subseteq\mathcal{P}$ be a set of precision variables in the directed acyclic graph $\mathcal{G} = (\mathcal{V}, \mathcal{E})$ with $A, Y \in \mathcal{V}$ for estimating the causal effect $\tau$ of $A$ on $Y$. Let $Z \subseteq \mathcal{V} \setminus \{A, Y \}$ and let $U \subseteq \mathcal{W}$ be a set of extended confounding variables, such that $Z \cup P \cup U$ is $Z \cup P$-irreducible for estimating $\tau$. The following d-separations hold in $\mathcal{G}'= \mathcal{G} \setminus (A \rightarrow Y)$: 
\begin{enumerate}[(i)]
    \item \label{P-d-sep-i}$A \indep_{\mathcal{G}'} Y \mid  Z \cup P \cup U$,
    \item \label{P-d-sep-ii}$A \indep_{\mathcal{G}'} Y \mid  Z  \cup U$.

\end{enumerate}
  Hence, the sets $Z \cup P \cup U$ and $Z \cup U$ are both valid adjustment sets. The following d-separations hold in the original graph $\mathcal{G}$:
    \begin{enumerate}[(a)]
        \item \label{d-sep-P-2} $P \indep_{\mathcal{G}} A \mid Z$,
        \item \label{d-sep-P-1} $P \indep_{\mathcal{G}} U \mid A \cup Z$.
    \end{enumerate}
    Let $S$ be a suboptimal precision variable in $\mathcal{G}$. Let $P^*$ be another precision variable, such that $S \indep_{\mathcal{G}'} Y \mid P^* \cup Z$ in $\mathcal{G}'$ for all $Z \subseteq \mathcal{V} \setminus \{A, Y \}$. Then, also the following d-separation holds in the original graph $\mathcal{G}$:
    \begin{enumerate}[(c)]
        \item \label{d-sep-P-4} $S \indep_{\mathcal{G}} Y \mid Z \cup P^*$.
    \end{enumerate}   
\end{lemma}
\begin{proof}
First, we show each d-separation in $\mathcal{G}'$ separately.
\begin{enumerate}[(i)]
    \item Holds because $Z \cup P \cup U$ is a valid adjustment set by Definition~\ref{irreducible}.
    \item This is implied from \ref{P-d-sep-i} $A \indep_{\mathcal{G}'} Y \mid  Z \cup P \cup U$ and $P \indep_{\mathcal{G}'} A \mid Z \cup U$, which holds by the definition of precision variables, by the contraction property \citep{forre2017markov, dawid1979}.
\end{enumerate}
Now we show each d-separation in $\mathcal{G}$ separately.
\begin{enumerate}[(a)]
    \item $P \indep_{\mathcal{G}'} A \mid Z$ holds by the definition of precision variables. Any path $\pi$ between a variable $P_i \in P$ and $A$ that is in $\mathcal{G}'$ but not in $\mathcal{G}$ must contain the edge $A \rightarrow Y$, such that $\pi$ contains the collider $Y$ and is blocked. It follows that $P \indep_{\mathcal{G}} A \mid Z$ also in $\mathcal{G}$.
    \item We show by contradiction that all paths between any variable $U_i \in U$ and any variable $P_i \in P$ in $\mathcal{G}$ must contain $Y$. Assume there exists a simple path $\pi = (U_i, \dots, P_i)$ in $\mathcal{G}$ with $Y \not \in \pi$. Then, there also exists a path $\pi' = (A, \dots, U_i, \dots, P_i)$ with $Y \not\in \pi'$, since $U_i$ is an extended confounding variable. It follows that $P_i \not\indep_{\mathcal{G}} A \mid X$, where $X \subseteq \mathcal{V} \setminus \{A, Y \}$ contains all colliders on the simple path derived from $\pi$. This contradicts \ref{d-sep-P-2}, and we conclude that all paths between a variable $U_i \in U$ and a variable $P_i \in P$ must contain $Y$. It follows that all paths between $U_i$ and $P_i$ are blocked given $A \cup Z$ for all $U_i \in U$ and all $P_i \in P$, and we conclude $P \indep_{\mathcal{G}} U \mid K \cup A$.  
    \item $S \indep_{\mathcal{G}'} Y \mid Z \cup P^*$ holds in $\mathcal{G}'$ by the definition of suboptimal precision variables. We show that it also holds in $\mathcal{G}$ by contradiction. Assume that $S \not\indep_{\mathcal{G}} Y \mid Z \cup P^*$ in $\mathcal{G}$. Any open path between $S$ and $Y$ given $Z \cup P^*$ must contain the edge $A \rightarrow Y$ and therefore the node $A$, since otherwise $S \not\indep_{\mathcal{G}'} Y \mid Z \cup P^*$. It follows that there is an open path between $S$ and $A$ given $Z \cup P^*$, which contradicts \ref{d-sep-P-2} with $P = \{S\}$ and $Z = Z \cup P^*$.
\end{enumerate}

\end{proof}

\begin{lemma}[d-separation properties of suboptimal confounding variables]
\label{lemma-d-separation-W}
    Let $S$ be a suboptimal confounding variable in a directed acyclic graph $\mathcal{G} = (\mathcal{V}, \mathcal{E})$ with $A, Y \in \mathcal{V}$ for estimating the causal effect $\tau$ of $A$ on $Y$. Let $W^*$ be another extended confounding variable, such that $S \indep_{\mathcal{G}'} Y \mid W^* \cup Z$ and $W^* \indep_{\mathcal{G}'} A \mid Z \cup S$, where $Z \subseteq \mathcal{V} \setminus \{A, Y \}$ and $\mathcal{G}' = \mathcal{G} \setminus (A \rightarrow Y)$. Let $U \subseteq \mathcal{W}$ be a set of extended confounding variables, such that $S \cup  W^* \cup Z \cup U$ is $S \cup  W^* \cup Z$-irreducible for estimating $\tau$. The following d-separations hold in $\mathcal{G}'$:
    \begin{enumerate}[(i)]
        \item \label{d-sep-i}$A \indep_{\mathcal{G}'} Y \mid S \cup W^* \cup Z \cup U$,
        \item \label{d-sep-ii}$A \indep_{\mathcal{G}'} Y \mid S \cup Z \cup U$,
        \item \label{d-sep-iii}$A \indep_{\mathcal{G}'} Y \mid W^* \cup Z \setminus S  \cup U$.
    \end{enumerate}
    Hence, the sets $S \cup Z \cup U$ and $W^* \cup Z \setminus S  \cup U$ are also valid adjustment sets. The following d-separations hold in the original graph $\mathcal{G}$:
    \begin{enumerate}[(a)]
        \item \label{d-sep-5} $W^* \indep_{\mathcal{G}} A \mid Z \cup S$,
        \item \label{d-sep-3} $S \indep_{\mathcal{G}} U \mid Z$,
        \item \label{d-sep-4} $W^* \indep_{\mathcal{G}} U \mid Z \setminus S$,
        \item \label{d-sep-1} $W^* \indep_{\mathcal{G}} U \mid A \cup Z \cup S$,
        \item \label{d-sep-2} $S \indep_{\mathcal{G}} Y \mid A \cup Z \setminus S \cup W^* \cup U$.
    \end{enumerate}
\end{lemma}
\begin{proof}
First, we show each d-separation in $\mathcal{G}'$ separately.
\begin{enumerate}[(i)]
    \item Holds because $S \cup  W^* \cup Z \cup U$ is a valid adjustment set by Definition~\ref{irreducible}.
    \item This is implied from $A \indep_{\mathcal{G}'} Y \mid S \cup W^* \cup Z \cup U$ and $W^* \indep_{\mathcal{G}'} A \mid Z \cup S \cup U$ by the contraction property \citep{forre2017markov, dawid1979}.
    \item This is implied from $A \indep_{\mathcal{G}'} Y \mid S \cup W^* \cup Z \setminus S  \cup U$ and $S \indep_{\mathcal{G}'} Y \mid W^* \cup Z \setminus S  \cup U$ by the contraction property.
\end{enumerate}
Now we show each d-separation in $\mathcal{G}$ separately.
\begin{enumerate}[(a)]
    \item Any path between $W^*$ and $A$ in $\mathcal{G}$ that is not in $\mathcal{G}'$ contains the edge $A \rightarrow Y$ with $Y$ as a collider and is therefore blocked given any set $X \subseteq \mathcal{V} \setminus\{A, Y\}$. It follows that $W^* \indep_{\mathcal{G}} A \mid Z \cup S$ in $\mathcal{G}$.
    \item Consider a path $\pi$ between $S$ and a variable $U_i \in U$ in $\mathcal{G}'$. We now show by contradiction that $\pi$ must contain $A$ or $Y$. Assume that $A, Y \not\in \pi$. 
    
    By the definition of extended confounding variables, there exists a path $\phi=(S, \dots, W^*)$ in $\mathcal{G}'$. By the definition of suboptimal confounding variables, we know that $A, Y \not\in \phi$, because otherwise, there would exist a set $Z \subseteq \mathcal{V} \setminus \{A, Y\}$, such that $S \not\indep_{\mathcal{G}'} Y \mid W^* \cup Z$ or $W^* \not\indep_{\mathcal{G}'} A \mid Z \cup S$. It follows that there either exists a simple path $\pi_1=(S, \dots, W^*, \dots,  U_i)$ or a simple path $\pi_2 = (W^*, \dots, S, \dots, U_i)$ with $A, Y \not\in \pi_1,\pi_2$.

    Since $U_i$ is $S \cup  W^* \cup Z$-irreducible, there exists a path $\pi_a=(U_i, \dots, A)$ with $Y \not\in \pi_a$ and a path $\pi_b=(U_i, \dots, Y)$ with $A \not\in \pi_b$. By concatenating these paths to $\pi_1$ and $\pi_2$, it follows that there exists either a path $\pi_1'= (S, \dots, W^*, \dots U_i, \dots, A)$ or a path $\pi_2'=(W^*, \dots, S, \dots, U_i, \dots, Y)$.

    Consider the case (1) where $\pi_1'$ exists. Within (1), consider the case (1a) where there exists at least one path of the form $\pi_a \subseteq \pi_1'$ that does not contain $S$ as a non-collider, i.e. it either contains $S$ as a collider or not at all. In this case, we have $W^* \not\indep_{\mathcal{G}'} A \mid S \cup X_1$, where $X_1 \subseteq \mathcal{V} \setminus \{A, Y\}$ consists of all colliders on the simple path between $W^*$ and $A$ derived from $\pi_1'$. This contradicts the definition of suboptimal confounding variables. Now consider the case (1b), where all paths of the form $\pi_a \subseteq \pi_1'$ contain $S$ as a non-collider. In this case, $S \cup  W^* \cup Z \cup U \setminus U_i$ is a valid adjustment set, as all paths between $U_i$ and $A$ are blocked given $S \cup  W^* \cup Z \cup U \setminus U_i$. Then, $S \cup  W^* \cup Z \cup U$ is not $S \cup  W^* \cup Z$-irreducible, which is a contradiction.

    Now consider the case (2) where $\pi_2'$ exists. Within (2), consider the case (2a) where there exists at least one path of the form $\pi_b \subseteq \pi_2'$ that does not contain $W^*$ as a non-collider, i.e. it either contains $W^*$ as a collider or not at all. In this case, we have $S \not\indep_{\mathcal{G}'} Y \mid W^* \cup X_2$, where $X_2 \subseteq \mathcal{V} \setminus \{A, Y\}$ consists of all colliders on the simple path between $S$ and $Y$ derived from $\pi_2'$. This contradicts the definition of suboptimal confounding variables. Now consider the case (2b), where all paths of the form $\pi_b \subseteq $ contain $W^*$ as a non-collider. In this case, $S \cup  W^* \cup Z \cup U$ is not $S \cup  W^* \cup Z$-irreducible, as all paths between $U_i$ and $Y$ are blocked given $S \cup W^* \cup Z \cup U \setminus U_i$. This contradicts our definition of $U$. 

    It follows that any path between $S$ and $U_i$ must contain $A$ or $Y$, and is therefore blocked in $\mathcal{G}'$ given any subset $X \subseteq \mathcal{V} \setminus\{A, Y\}$. Any path in $\mathcal{G}$ that is not in $\mathcal{G}'$ contains the edge $A \rightarrow Y$ and is therefore also blocked given any subset $X \subseteq \mathcal{V} \setminus\{A, Y\}$, as $Y$ is a collider on the path. We conclude that $S \indep_{\mathcal{G}} U \mid X$ holds in $\mathcal{G}$ for any subset $X \subseteq \mathcal{V} \setminus\{A, Y\}$.
    \item  We set $Z = Z \setminus S$ and consider a path $\pi$ between $W^*$ and a variable $U_i \in U$ in $\mathcal{G}'$, and apply the proof of \ref{d-sep-3} without further changes, arriving at the conclusion that $W^* \indep_{\mathcal{G}} U \mid X$ for any subset $X \subseteq \mathcal{V} \setminus \{A, Y\}$.
    
    \item From the proof of \ref{d-sep-4}, we know that $W^* \indep_{\mathcal{G}} U \mid Z \cup S$ and $W^* \indep_{\mathcal{G}'} U \mid Z \cup S$. We also know that $W^* \indep_{\mathcal{G}'} A \mid Z \cup S$, which makes $W^* \indep_{\mathcal{G}'} A \cup U \mid Z \cup S$ by composition. By the weak union property \citep{forre2017markov, dawid1979}, it follows that $W^* \indep_{\mathcal{G}'} U \mid A \cup Z \cup S$ in $\mathcal{G}'$. Hence, $W^* \indep_{\mathcal{G}} U \mid A \cup Z \cup S$ also in $\mathcal{G}$, as any path between $W^*$ and $U$ in $\mathcal{G}$ that is not in $\mathcal{G}'$ contains $Y$ as a collider.
    
    \item By the definition of suboptimal confounding variables, we have $S \indep_{\mathcal{G}'} Y \mid W^* \cup Z \setminus S \cup U$ in $\mathcal{G}'$ and by \ref{d-sep-iii}, we have $A \indep_{\mathcal{G}'} Y \mid W^* \cup Z \setminus S \cup U$ in $\mathcal{G}'$. By the weak union property, it follows that $S \indep_{\mathcal{G}'} Y \mid A \cup W^* \cup Z \setminus S \cup U$. Hence, it also holds that $S \indep_{\mathcal{G}} Y \mid A \cup W^* \cup Z \setminus S \cup U$, since any path between $S$ and $Y$ containing the edge $A \rightarrow Y$ is blocked given $A \cup W^* \cup Z \setminus S \cup U$. 
\end{enumerate}
\end{proof}

\begin{lemma}[d-separation properties of irrelevant variables]
\label{lemma-d-separation-I}
Let $I \subseteq\mathcal{I}$ be a set of irrelevant variables in the directed acyclic graph $\mathcal{G} = (\mathcal{V}, \mathcal{E})$ with $A, Y \in \mathcal{V}$ for estimating the causal effect $\tau$ of $A$ on $Y$. Let $Z \subseteq \mathcal{V} \setminus \{A, Y \}$ and let $U \subseteq \mathcal{W}$ be a set of extended confounding variables, such that $Z \cup I \cup U$ is $Z \cup I$-irreducible for estimating $\tau$. The following d-separations hold in $\mathcal{G}'= \mathcal{G} \setminus (A \rightarrow Y)$: 
\begin{enumerate}[(i)]
    \item \label{I-d-sep-i}$A \indep_{\mathcal{G}'} Y \mid  Z \cup I \cup U$,
    \item \label{I-d-sep-ii}$A \indep_{\mathcal{G}'} Y \mid  Z  \cup U$.

\end{enumerate}
  Hence, the sets $Z \cup I \cup U$ and $Z \cup U$ are both valid adjustment sets. The following d-separations hold in the original graph $\mathcal{G}$:
    \begin{enumerate}[(a)]
        \item \label{d-sep-I-1} $I \indep_{\mathcal{G}} Y \mid A \cup Z \cup U$,
        \item \label{d-sep-I-2} $I \indep_{\mathcal{G}} U \mid Z$.
    \end{enumerate}
\end{lemma}
\begin{proof}
First, we show each d-separation in $\mathcal{G}'$ separately.
\begin{enumerate}[(i)]
    \item Holds because $Z \cup I \cup U$ is a valid adjustment set.
    \item This is implied from \ref{I-d-sep-i} $A \indep_{\mathcal{G}'} Y \mid  Z \cup I \cup U$ and $I \indep_{\mathcal{G}'} Y \mid Z \cup U$, which holds by the definition of irrelevant variables, by the contraction property \citep{forre2017markov, dawid1979}.
\end{enumerate}
Now we show each d-separation in $\mathcal{G}$ separately.
\begin{enumerate}[(a)]
    \item By \ref{I-d-sep-ii}, it holds that $A \indep_{\mathcal{G}'} Y \mid  Z  \cup U$. By the definition of irrelevant variables, it also holds that $I \indep_{\mathcal{G}'} Y \mid Z \cup U$. It follows that $A \cup I \indep_{\mathcal{G}'} Y \mid Z \cup U$ by the composition property \citep{forre2017markov, dawid1979} and $I \indep_{\mathcal{G}'} Y \mid A \cup Z \cup U$ by the weak union property \citep{forre2017markov, dawid1979}. Any path $\pi$ between a variable $I_i \in I$ and $Y$ in $\mathcal{G}$ that is not in $\mathcal{G}'$ must contain the edge $A \rightarrow Y$. Since $A$ has only incoming edges apart from its edge to $Y$, it follows that $\pi$ is blocked given $A \cup Z \cup U$. It follows that $I \indep_{\mathcal{G}} Y \mid A \cup Z \cup U$.
    \item First, we show by contradiction that any path in $\mathcal{G}'$ between any variable $I_i \in I$ and any variable $U_i \in U$ contains $A$. Assume that there exists a path $\pi$ between a variable $I_i \in I$ and a variable $U_i \in U$ with $A \not\in \pi$. Then, there also exists a path $\pi'=(I_i, \dots, U_i, \dots, Y)$ with $A \not\in \pi'$, because $U_i$ is an extended confounding variable. Furthermore, there exists a path of the form $\pi'$ with $A \not\in \pi'$, such that $\pi'$ does not contain any $Z_i \in Z$ as a non-collider, because otherwise, $Z \cup I \cup U\setminus U_i$ would be a valid adjustment set, and $Z \cup I \cup U$ would not be $Z \cup I$-irreducible. It follows that $I_i \not\indep_{\mathcal{G}'} Y \mid X \cup Z$, where $X \subseteq \mathcal{V} \setminus \{A, Y\}$ consists of all colliders on the simple path derived from $\pi'$. This contradicts the definition of irrelevant variables. It follows that all paths between a variable $I_i \in I$ and a variable $U_i \in U$ must contain $A$ and are therefore blocked in $\mathcal{G}'$. Any path $\phi$ between a variable $I_i \in I$ and a variable $U_i \in U$ in $\mathcal{G}$ that is not in $\mathcal{G}'$ must contain the edge $A \rightarrow Y$. It follows that $\phi$ is blocked since $Y \in\phi$, which is a collider on $\phi$. Hence, we have $I \indep_{\mathcal{G}} U \mid Z$.
\end{enumerate}
\end{proof}

\subsection*{Condition for inclusion of precision variables}
\label{A:cor:precision}
\begin{corollary}[Condition for Inclusion of Precision Variables] \label{cor:precision}Let $K$ be an adjustment set for estimating the causal effect $\tau$ and $P\subseteq \mathcal{P}$ be a set of precision variables, where $P \cap K = \emptyset$. We assume $|K \cup P| < n-3$. A necessary and sufficient condition to ensure a lower expected mean squared error of the ordinary least squares estimator $\hat{\tau}_{K \cup P}$ compared to $\hat{\tau}_{K}$, is
\begin{equation}
    \frac{|P|}{n-|K|-3} < 1 - \frac{\sigma_{yy.akp}}{\sigma_{yy.ak}}.
\end{equation}
\end{corollary}

\begin{proof}
First, in Lemma~\ref{lemma:bias}, we show that the bias of the ordinary least squares estimator $\hat \tau_{K}$ is invariant to the addition of a precision variable to the adjustment set $K$. While previous works \citep{pearl2009causality,hernan2024causal} have established this property for valid adjustment sets, we extend the result to cases where $K$ may not satisfy the criteria for validity. This demonstrates that the inclusion of variables independent of $A$ given $K$ does not affect the bias of the estimator, regardless of the validity of the remaining adjustment set.
\begin{lemma}[Bias invariance]
    \label{lemma:bias}  
Let $\mathcal{V}$ be the set of variables in a directed acyclic graph $\mathcal{G}=(\mathcal{V}, \mathcal{E})$ describing the causal relations of $\mathcal{V}$, with $A, Y \in \mathcal{V}$. Let $K \subseteq \mathcal{V} \setminus \{A, Y\}$ be a set of variables and let $P$ be a precision variable in $\mathcal{G}$. The bias of $\hat \tau_{K}$ is invariant to the addition of $P$:
\begin{equation*}
   B(\hat \tau_{K}) = B(\hat \tau_{K \cup P})
\end{equation*}
\end{lemma}
\begin{proof}
    We base the proof on a more general version of the calculations in \cite{pearl2010onAClass}. We need to compare the following quantities, where $k$ is a vector of values of all variables $V_i \in K$, and $a$ and $p$ denote the values of $A$ and $P$ respectively:
    \begin{align}
        E(\hat \tau_{K}) &= \frac{\partial}{\partial a}E(Y \mid a, k)
        \label{eq:expected-ate-1}, \\
        E(\hat \tau_{K \cup P}) &= \frac{\partial}{\partial a}E(Y \mid a, k, p). 
        \label{eq:expected-ate-2}
    \end{align}
    For the bias to be invariant, it needs to hold that:
    \begin{align}
        E(\hat \tau_{K}) &= E(\hat \tau_{K \cup P}).
        \label{eq:equal-estimates}
    \end{align}
    Let $U$ be a potentially empty set of covariates, such that $K \cup P \cup U$ is a $K \cup P$-irreducible adjustment set. Such a set $U$ exists for all $K \subseteq \mathcal{V} \setminus \{A, Y\}$ by Proposition~\ref{prop:irreducible}. By Lemma~\ref{lemma-d-separation-P} (ii), $K \cup U$ is also a valid adjustment set. Using Equation~\eqref{eq:expected-ate-1} and \eqref{eq:expected-ate-2}, it follows that
    \begin{align*}
        E(\hat \tau_{K \cup P}) - E(\hat \tau_{K}) &= \frac{\partial}{\partial a} \left\{ \int_u E(Y \mid a, k, p, u) \mathrm{pr}(u \mid a,k, p) du - \int_u E(Y \mid a, k, u) \mathrm{pr}(u \mid a,k) du  \right\}.
    \end{align*}
    By Lemma~\ref{lemma-d-separation-P}, we have $U \indep_{\mathcal{G}} P \mid K \cup A$. Therefore, $\mathrm{pr}(u \mid a,k,p)=\mathrm{pr}(u \mid a,k)$ and    
    \begin{align*}
        E(\hat \tau_{K \cup P}) - E(\hat \tau_{K}) &= \frac{\partial}{\partial a} \left[ \int_u \left\{ E(Y \mid a, k, p, u) - E(Y \mid a, k, u)\right\} \mathrm{pr}(u \mid a,k) du  \right].
    \end{align*}
    By linearity of $Y$, this equals
    \begin{align*}
        E(\hat \tau_{K \cup P}) - E(\hat \tau_{K}) &= \frac{\partial}{\partial a} \left\{ \int_u \beta_{yp.aku}p \mathrm{pr}(u \mid a,k) du  \right\},
    \end{align*}
    where $\beta_{yp.aku}$ is the coefficient vector describing the effects of $P$ on $Y$ from regressing $Y$ on $P$, $A$, $K$ and $U$. It follows:
    \begin{align*}
        E(\hat \tau_{K \cup P}) - E(\hat \tau_{K}) &= \frac{\partial}{\partial a} \beta_{yp.aku}p = 0,
    \end{align*}
    and hence Equation~\eqref{eq:equal-estimates} holds.
    \end{proof}
\label{app:proof-precision}
    We proceed with the proof of Corollary~\ref{cor:precision}. Using Lemma~\ref{lemma:bias} and Equation~\eqref{eq:ssc-rewritten} from the proof of Theorem~1, it is sufficient to show the following:

   \begin{align*}
    0 &< \frac{1}{n - |K | - 3} \frac{\sigma_{yy.ak}}{\sigma_{aa.k}}  - \frac{1}{n - |K  \cup P| - 3}\frac{\sigma_{yy.akp}}{\sigma_{aa.kp}} 
    \end{align*}
    By Lemma~\ref{lemma-d-separation-P}, we have $A \indep_{\mathcal{G}} P \mid K$. By applying Lemma~\ref{lemma:equal-coeffs} \citep{henckelGraphicalCriteriaEfficient2022a} in conjunction with the causal Markov property with $T=P$, $Y=A$, $X=K$ and $S=\emptyset$, we have $\sigma_{aa.k} = \sigma_{aa.kp}$, such that it remains to show that
    \begin{align*}
        \frac{\sigma_{yy.akp}}{n - |K \cup P| - 3}  &< \frac{\sigma_{yy.akp}}{n - |K| - 3}.
    \end{align*}
    This can be rewritten as follows:
    \begin{align}
        1 - \frac{\sigma_{yy.akp}}{\sigma_{yy.ak}}  &>  \frac{|P|}{n - |K| - 3}.
        \label{eq:precision-proof-statement}
    \end{align}
    Equation~\eqref{eq:precision-proof-statement} holds by the criterion for inclusion of precision variables. 
\end{proof}

\subsection*{Proof of Theorem 2 (MSE-optimal adjustment set candidates)}
For the proof of the first graphical criterion, we will consider precision variables, extended confounding variables and all remaining pre-treatment variables separately.

\begin{lemma}[Exclusion of suboptimal precision variables] 
\label{lemma:precision}
Any variable that is a suboptimal precision variable $S \in \mathcal{S}^P$ does not need to be considered for the MSE-optimal adjustment set in the sense that for any adjustment set $K \cup S \subseteq \mathcal{V} \setminus \{A, Y\}$, there exists a non-suboptimal precision variable $P^*$, such that the adjustment set $K \setminus S \cup P^*$ yields a lower or equal mean squared error. Hence, there exists an MSE-optimal adjustment set $O^*_n$, such that
\begin{equation}
    O^*_n \cap \mathcal{S}^P = \emptyset.
\end{equation}
\end{lemma}
\begin{proof}
    Consider a suboptimal precision variable $S \in \mathcal{S}^P$ and any adjustment set $K \subseteq \mathcal{V} \setminus \{A, Y\}$. By the definition of suboptimal precision variables, there exists another precision variable $P^* \in \mathcal{P}$, such that all paths from $S$ to $Y$ in $\mathcal{G}'$ are blocked given $P^*$ and $K$. The estimator $\hat \tau_{K \cup S}$ using the adjustment set $K \cup S$ has variance
    \begin{align*}
        \text{var}(\hat \tau_{K \cup S}) &= \frac{1}{n - |K \cup S| - 3}\frac{\sigma_{yy.aks}}{\sigma_{aa.ks}}.
    \end{align*}The estimator $\hat \tau_{K \setminus S \cup P^*}$ using the adjustment set $K \setminus S \cup P^*$ has variance
    \begin{align}
        \text{var}(\hat \tau_{K \setminus S \cup P^*}) &= \frac{1}{n - |K \setminus S \cup P^*| - 3}\frac{\sigma_{yy.ak\setminus s p^*}}{\sigma_{aa.k\setminus s p^*}} \nonumber \\
        &= \frac{1}{n - (|K \cup S \cup P^*| - 1) - 3}\frac{\sigma_{yy.ak\setminus s p^*}}{\sigma_{aa.k\setminus s p^*}} \nonumber \\
        &\leq \frac{1}{n - |K \cup S| - 3}\frac{\sigma_{yy.ak\setminus s p^*}}{\sigma_{aa.k\setminus s p^*}}.
        \label{eq:var-precision}
    \end{align}
    By Lemma~\ref{lemma-d-separation-P} (c), we have $S \indep_{\mathcal{G}} Y \mid K \setminus S \cup P^*$ in $\mathcal{G}$. By Lemma~\ref{lemma-d-separation-P} (a), we also have $P^* \indep_{\mathcal{G}} A \mid K \cup S$ in $\mathcal{G}$. By applying Lemma~\ref{lemma:ci-variance} \citep{henckelGraphicalCriteriaEfficient2022a} with $X=A$, $T = S$, $S = P^*$ and $W = K \setminus S$ in conjunction with the causal Markov property, it follows that
    \begin{align}
    \label{eq:a17}
        \frac{\sigma_{yy.ak\setminus s p^*}}{\sigma_{aa.k\setminus s p^*}} \leq \frac{\sigma_{yy.aks}}{\sigma_{aa.ks}}, 
    \end{align}
    and hence
        $\text{var}(\hat \tau_{K \setminus S \cup P^*}) \leq  \text{var}(\hat \tau_{K \cup S})$. Using Lemma~\ref{lemma:bias}, the bias is unaffected by both $S$ and $P^*$, i.e.
        $B^2(\hat \tau_{K \setminus S\cup P^*}) = B^2(\hat \tau_{K \cup S})$, and therefore $\text{MSE}(\hat \tau_{K \setminus S \cup P^*}) \leq  \text{MSE}(\hat \tau_{K \cup S})$. 
\end{proof}

\begin{lemma}[Exclusion of suboptimal confounding variables] 
\label{lemma:confounding}Any variable that is a suboptimal confounding variable $S \in \mathcal{S}^W$ does not need to be considered for the MSE-optimal adjustment set in the sense that for any adjustment set $K \cup S \subseteq \mathcal{V} \setminus \{A, Y\}$, there exists a non-suboptimal extended confounding variable $W^*$, such that the adjustment set $K \setminus S \cup W^*$ yields a lower or equal mean squared error.  Hence, there exists an MSE-optimal adjustment set $O^*_n$, such that
\begin{equation}
    O^*_n \cap \mathcal{S}^W = \emptyset.
\end{equation}
\end{lemma}
\begin{proof}
    Consider a suboptimal confounding variable $S \in \mathcal{S}^W$ and any adjustment set $K \subseteq \mathcal{V} \setminus \{A, Y\}$. The estimator $\hat \tau_{K \cup S}$ using the adjustment set $K \cup S$ has variance
    \begin{align}
        \text{var}(\hat \tau_{K \cup S}) &= \frac{1}{n - |K \cup S| - 3}\frac{\sigma_{yy.aks}}{\sigma_{aa.ks}}.
    \end{align}
    By the definition of suboptimal confounding variables, there exists another confounding variable $W^* \in \mathcal{W}$, such that all paths between $S$ and $Y$ are blocked in $\mathcal{G}'=\mathcal{G}\setminus (A \rightarrow Y)$ given $W^*$ and any other set $Z \subseteq \mathcal{V} \setminus \{A, Y\}$, and $S$ blocks all paths between $W^*$ and $A$ in $\mathcal{G}'$ given $Z$. The estimator $\hat \tau_{K \setminus S \cup W^*}$ using the adjustment set $K \setminus S \cup W^*$ has variance
    \begin{align}
        \text{var}(\hat \tau_{K \setminus S \cup W^*}) &= \frac{1}{n - |K \setminus S \cup W^*| - 3}\frac{\sigma_{yy.ak\setminus s w^*}}{\sigma_{aa.k\setminus s w^*}}.
    \end{align}
For the bias of the estimator $\hat \tau_{K \cup S}$ it holds that
    \begin{align}
        B(\hat \tau_{K \cup S}) &= E\left( \hat \tau_{K \cup S} - \tau \right) = \frac{\partial}{\partial a}\left\{ \int_{u} E (Y \mid a, k, s, u) \mathrm{pr}(u \mid a, k, s) du\right\} - \tau, 
        \label{eq:suboptimal-bias-start}
    \end{align}
    where $U$ is a potentially empty set, such that $K \cup S \cup U$ is a $K \cup S$-irreducible adjustment set. By Lemma~\ref{lemma-d-separation-W} (iii), $K\setminus S \cup W^* \cup U$ is also a valid adjustment set. If $U$ is empty, this means that $K \cup S$ and $K\setminus S \cup W^*$ are already valid adjustment sets, such that the bias is zero. By linearity,

    \begin{align}
        B(\hat \tau_{K \cup S}) &=  \frac{\partial}{\partial a} \left\{\left( \tau a + \beta_{yk.asu} k + \beta_{ys.aku} s + \mu_{y.aksu} \right) \int_u \mathrm{pr}(u \mid a,k,s) +  \int_u \beta_{yu.aks} u \mathrm{pr}(u \mid a,k,s) \right\} - \tau \nonumber \\
        &=  \beta_{yu.aks} \frac{\partial}{\partial a} E(U \mid a,k,s). 
    \end{align}
    By linearity and Gaussianity of the variables,
    \begin{align}
        B(\hat \tau_{K \cup S}) &= \beta_{yu.aks} \frac{\partial}{\partial a} \left( \beta_{ua.ks} a + \beta_{uk.as} k + \beta_{ua.ak} s + \mu_{u.aks} \right)= \beta_{yu.aks} \beta_{ua.ks} \nonumber \\
        &= \beta_{yu.aks} \frac{\sigma_{ua.ks}}{\sigma_{aa.ks}}. 
    \end{align}
    Similarly,
    \begin{align}
        B(\hat \tau_{K \setminus S \cup W^*}) &= \beta_{yu.ak\setminus sw^*} \frac{\sigma_{ua.k \setminus s w^*}}{\sigma_{aa.k \setminus sw^*}}. 
        \label{eq:suboptimal-bias-end}
    \end{align}
   The MSE of $\hat \tau_{K \cup S}$ is:
   \begin{align}
   \label{mse-invalid-set}
       \text{MSE}(\hat \tau_{K \cup S}) &= \left(\frac{\beta_{yu.aks} \sigma_{ua.ks}}{\sigma_{aa.ks}} \right)^2 + \frac{\sigma_{yy.aks}}{(n - |K \cup S| - 3)\sigma_{aa.ks}}. 
   \end{align}
   Since $|K \cup S| \geq |K \setminus S \cup W^*|$, it follows that
   \begin{align}
       \text{MSE}(\hat \tau_{K \cup S}) &\geq \left(\frac{\beta_{yu.aks} \sigma_{ua.ks}}{\sigma_{aa.ks}} \right)^2 + \frac{\sigma_{yy.aks}}{(n - |K \setminus S \cup W^*| - 3)\sigma_{aa.ks}}.
    \end{align}
    By Lemma C.4 from \citep{henckelGraphicalCriteriaEfficient2022a}, we have $\sigma_{yy.aks} = \sigma_{yy.aksu} + \beta_{yu.aks} \Sigma_{uu.aks}\beta_{yu.aks}^T$, and therefore 
    \begin{align}
    \label{w-mse-last-step-1}
     \text{MSE}(\hat \tau_{K \cup S})  &\geq \frac{1}{\sigma_{aa.ks}} \beta_{yu.aks} \Psi(\hat \tau_{K \cup S}) \beta_{yu.aks}^T + \frac{\sigma_{yy.aksu}}{(n - |K \setminus S \cup W^*| - 3)\sigma_{aa.ks}}, 
    \end{align}
    where
    \begin{align*}
       \Psi(\hat \tau_{K \cup S}) &= \frac{\sigma_{ua.ks} \sigma_{ua.ks}^T}{\sigma_{aa.ks}} + \frac{\Sigma_{uu.aks}}{n - |K \cup S| - 3}\geq \frac{\sigma_{ua.ks} \sigma_{ua.ks}^T}{\sigma_{aa.ks}} + \frac{\Sigma_{uu.aks}}{n - |K \setminus S \cup W^*| - 3}.
    \end{align*}
    By the Schur complement formula, it holds that $\Sigma_{uu.aks}=\Sigma_{uu.ks} - \frac{\sigma_{ua.ks}\sigma_{ua.ks}^T}{\sigma_{aa.ks}}$, such that
    \begin{align}
        \Psi(\hat \tau_{K \cup S}) &\geq \frac{\sigma_{ua.ks} \sigma_{ua.ks}^T}{\sigma_{aa.ks}} + \frac{\Sigma_{uu.ks} - \frac{\sigma_{ua.ks}\sigma_{ua.ks}^T}{\sigma_{aa.ks}}}{n - |K \setminus S \cup W^*| - 3}, \nonumber \\
        &= \frac{\sigma_{ua.ks} \sigma_{ua.ks}^T (n - |K \setminus S \cup W^*| - 3) +\Sigma_{uu.ks} \sigma_{aa.ks}- \sigma_{ua.ks}\sigma_{ua.ks}^T}{\sigma_{aa.ks}(n - |K \setminus S \cup W^*| - 3)} , \nonumber \\
        &= \frac{\sigma_{ua.ks} \sigma_{ua.ks}^T (n - |K \setminus S \cup W^*| - 4) +\Sigma_{uu.ks} \sigma_{aa.ks}}{\sigma_{aa.ks}(n - |K \setminus S \cup W^*| - 3)}, \nonumber \\ 
        &=  \frac{\sigma_{ua.ks} \sigma_{ua.ks}^T (n - |K \setminus S \cup W^*| - 4)}{\sigma_{aa.ks}(n - |K \setminus S \cup W^*| - 3)} + \frac{\Sigma_{uu.ks}}{n - |K \setminus S \cup W^*| - 3}.
        \label{w-mse-last-step-2}
    \end{align}
Similarly,
    \begin{align*}
        \text{MSE}(\hat \tau_{K \setminus S \cup W^*}) 
        &= \frac{1}{\sigma_{aa.k \setminus s w^*}} \beta_{yu.ak \setminus s w^*} \Psi(\hat \tau_{K \setminus S \cup W^*}) \beta_{yu.ak \setminus s w^*}^T  + \frac{\sigma_{yy.ak \setminus s w^*u}}{(n - |K \setminus S \cup W^*| - 3)\sigma_{aa.k \setminus s w^*}},
    \end{align*}
where
\begin{align*}
    \Psi(\hat \tau_{K \setminus S \cup W^*}) &= \frac{\sigma_{ua.k \setminus s w^*} \sigma_{ua.k \setminus s w^*}^T (n - |K \setminus S \cup W^*| - 4)}{\sigma_{aa.k \setminus s w^*}(n - |K \setminus S \cup W^*| - 3)} + \frac{\Sigma_{uu.k \setminus s w^*}}{n - |K \setminus S \cup W^*| - 3} .
\end{align*}
We now show that (1) $\sigma_{aa.ks} \leq \sigma_{aa.k \setminus s w^*}$, (2) $\beta_{yu.aks} = \beta_{yu.a\setminus s w^*}$, (3) $\sigma_{ua.ks} = \sigma_{ua.k \setminus s w^*}$, (4) $\Sigma_{uu.ks} = \Sigma_{uu.k \setminus s w^*}$ and (5) $\sigma_{yy.aksu} \geq \sigma_{yy.ak \setminus s w^*u}$. For this, we use the d-separation properties of suboptimal confounding variables from Lemma~\ref{lemma-d-separation-W}. 

By applying Lemma~\ref{lemma-d-separation-W} (a) and Corollary~\ref{cor:ci-variance-separated} (a) in conjunction with the causal Markov property, we have $\sigma_{aa.ks} \leq \sigma_{aa.k \setminus s w^*}$ if $W^* \not\in K$. If $W^* \in K$, we have $\sigma_{aa.ks} = \sigma_{aa.ks w^*} \leq \sigma_{aa.k \setminus s w^*}$ by the law of total variance. It follows that (1) $\sigma_{aa.ks} \leq \sigma_{aa.k \setminus s w^*}$.

By applying Lemma~\ref{lemma-d-separation-W} (d) and Lemma~\ref{lemma:equal-coeffs} \citep{henckelGraphicalCriteriaEfficient2022a} in conjunction with the causal Markov property, we have $\beta_{yu.aks} = \beta_{yu.aksw^*}$. By applying Lemma~\ref{lemma-d-separation-W} (e) and Lemma~\ref{lemma:equal-coeffs} \citep{henckelGraphicalCriteriaEfficient2022a} in conjunction with the causal Markov property, we have $\beta_{yu.aksw^*} = \beta_{yu.ak\setminus s w^*}$. It follows that (2) $\beta_{yu.aks} =  \beta_{yu.a\setminus s w^*}$. 

By applying Lemma~\ref{lemma-d-separation-W} (a) and Lemma~\ref{lemma:equal-cov} \citep{pmlr-v213-pena23a} in conjunction with the causal Markov property, we have $\sigma_{ua.ks} = \sigma_{ua.ksw^*}$. By applying Lemma~\ref{lemma-d-separation-W} (b) with $Z=K \setminus S \cup W^*$ and Lemma~\ref{lemma:equal-cov} \citep{pmlr-v213-pena23a} in conjunction with the causal Markov property, we have $\sigma_{ua.ksw^*} = \sigma_{ua.k \setminus sw^*}$. It follows that (3) $\sigma_{ua.ks} = \sigma_{ua.k \setminus s w^*}$. 

By applying Lemma~\ref{lemma-d-separation-W} (c) with $Z = K \cup S$ and Lemma~\ref{lemma:equal-coeffs} \citep{henckelGraphicalCriteriaEfficient2022a} in conjunction with the causal Markov property, we have $\Sigma_{uu.ks} = \Sigma_{uu.ksw^*}$. By applying Lemma~\ref{lemma-d-separation-W} (b) with $Z = K \setminus S$ and Lemma~\ref{lemma:equal-coeffs} \citep{henckelGraphicalCriteriaEfficient2022a} in conjunction with the causal Markov property, we have $\Sigma_{uu.ksw^*} = \Sigma_{uu.k \setminus sw^*}$. It follows that (4) $\Sigma_{uu.ks} = \Sigma_{uu.k \setminus s w^*}$.

By applying Lemma~\ref{lemma-d-separation-W} (e) and Corollary~\ref{cor:ci-variance-separated} (b) in conjunction with the causal Markov property, we have $\sigma_{yy.aksu} \geq \sigma_{yy.ak \setminus s w^*u}$ if $W^* \not\in K$. If $W^* \in K$, we have $\sigma_{yy.aksu} \geq \sigma_{yy.ak \setminus s w^*u} = \sigma_{yy.ak \setminus su}$ by the law of total variance. It follows that (5) $\sigma_{yy.aksu} \geq \sigma_{yy.ak \setminus s w^*u}$.

Now we can apply (1)\textendash(5) to Equation~\eqref{w-mse-last-step-1} and  Equation~\eqref{w-mse-last-step-2}, which gives $\Psi(\hat \tau_{K \cup S }) \geq\Psi(\hat \tau_{K \setminus S \cup W^*})$, and
\begin{align*}
        \text{MSE}(\hat \tau_{K \cup S})  &\geq \frac{1}{\sigma_{aa.k \setminus s w^*}} \beta_{yu.ak \setminus s w^*} \Psi(\hat \tau_{K \setminus S \cup W^*}) \beta_{yu.ak \setminus s w^*}^T  + \frac{\sigma_{yy.ak \setminus s w^*u}}{(n - |K \setminus S \cup W^*| - 3)\sigma_{aa.k \setminus s w^*}^2}, \\
        &= \text{MSE}(\hat \tau_{K \setminus S \cup W^*}),      
\end{align*}
and this concludes the proof.
\end{proof}

\begin{lemma}[Exclusion of Irrelevant Variables]
\label{lemma:irrelevant}Any variable that is an irrelevant variable $I \in \mathcal{I}$ should not be considered for the MSE-optimal adjustment set in the sense that for any adjustment set $K \cup I$, the adjustment set $K \setminus I$ provides a lower MSE:
\begin{equation*}
    O_n \cap \mathcal{I} = \emptyset
\end{equation*}
\end{lemma}
\begin{proof}
    Consider an irrelevant variable $I \in \mathcal{I}$ and any adjustment set $K \subseteq \mathcal{V}$. By Equation~\eqref{mse-invalid-set}, the estimator $\hat \tau_{K \cup I}$ has the MSE
    \begin{align}
    \label{eq:irrelevant-first}
       \text{MSE}(\hat \tau_{K \cup I}) &= \left(\frac{\beta_{yu.aki} \sigma_{ua.ki}}{\sigma_{aa.ki}} \right)^2 + \frac{\sigma_{yy.aki}}{(n - |K \cup I| - 3)\sigma_{aa.ki}}, 
   \end{align}
    where $U$ is a potentially empty set, such that $K \cup I \cup U$ is a $K \cup I$-irreducible adjustment set. By Lemma~\ref{lemma-d-separation-I} (ii), $K \cup U$ is also a valid adjustment set. Since $|K \cup I|> |K \setminus I|$, it follows that
    \begin{align}
       \text{MSE}(\hat \tau_{K \cup I}) &> \left(\frac{\beta_{yu.aki} \sigma_{ua.ki}}{\sigma_{aa.ki}} \right)^2 + \frac{\sigma_{yy.aki}}{(n - |K \setminus I| - 3)\sigma_{aa.ki}}. 
   \end{align}
   From Equations~\eqref{w-mse-last-step-1}\textendash\eqref{w-mse-last-step-1}, it follows that 
   \begin{align}
   \label{i-mse-last-step}
        \text{MSE}(\hat \tau_{K \cup I})  &>\frac{1}{\sigma_{aa.ki}} \beta_{yu.aki} \Psi(\hat \tau_{K \cup I}) \beta_{yu.aki}^T  + \frac{\sigma_{yy.akiu}}{(n - |K \setminus I| - 3)\sigma_{aa.ki}},
    \end{align}
    where
    \begin{align}
         \Psi(\hat \tau_{K \cup I}) & >\frac{\sigma_{ua.ki} \sigma_{ua.ki}^T (n - |K \setminus I | - 4)}{\sigma_{aa.ki}(n - |K \setminus I| - 3)} + \frac{\Sigma_{uu.ki}}{n - |K \setminus I| - 3}.
    \end{align}
    Similarly,
    \begin{align}
        \text{MSE}(\hat \tau_{K \setminus I})  &=\frac{1}{\sigma_{aa.k \setminus i}} \beta_{yu.ak \setminus i} \Psi(\hat \tau_{K \setminus I}) \beta_{yu.ak \setminus i }^T  + \frac{\sigma_{yy.ak \setminus i u}}{(n - |K \setminus I| - 3)\sigma_{aa.k \setminus i}},
    \end{align}
    where
    \begin{align}
        \Psi(\hat \tau_{K \setminus I}) =  \frac{\sigma_{ua.k \setminus i} \sigma_{ua.k \setminus i}^T (n - |K \setminus I | - 4)}{\sigma_{aa.k \setminus i}(n - |K \setminus I| - 3)} + \frac{\Sigma_{uu.k \setminus i }}{n - |K \setminus I| - 3}. 
    \end{align}
    We now show that (1) $\sigma_{aa.ki} \leq \sigma_{aa.k \setminus i }$, (2) $\beta_{yu.aki} = \beta_{yu.ak\setminus i}$, (3) $\sigma_{ua.ki} = \sigma_{ua.k \setminus i }$, (4) $\Sigma_{uu.ki} = \Sigma_{uu.k \setminus i }$ and (5) $\sigma_{yy.akiu} = \sigma_{yy.ak \setminus i u}$. For this, we use the d-separation properties of irrelevant variables from Lemma~\ref{lemma-d-separation-I}. 

    From the law of total variance, it follows that (1) $\sigma_{aa.ki} \leq \sigma_{aa.k \setminus i }$. By Lemma~\ref{lemma-d-separation-I}, we have $I \indep_{\mathcal{G}} Y \mid A \cup U \cup K$ and $I \indep_{\mathcal{G}} U \mid K$. By applying Lemma~\ref{lemma:equal-coeffs} in conjunction with the causal Markov property, we have (2) $\beta_{yu.aki} = \beta_{yu.ak\setminus i}$, (4) $\Sigma_{uu.ki} = \Sigma_{uu.k \setminus i }$ and (5) $\sigma_{yy.akiu} = \sigma_{yy.ak \setminus i u}$. By applying Lemma~\ref{lemma:equal-cov} in conjunction with the causal Markov property, we have (3) $\sigma_{ua.ki} = \sigma_{ua.k \setminus i }$.

Now we can apply (1)\textendash(5) to Equation~\eqref{i-mse-last-step}, which gives $\Psi(\hat \tau_{K \cup I }) >\Psi(\hat \tau_{K \setminus I})$, and
\begin{align}
\label{eq:irrelevant-last}
        \text{MSE}(\hat \tau_{K \cup I})  &> \frac{1}{\sigma_{aa.k \setminus i}} \beta_{yu.ak \setminus i} \Psi(\hat \tau_{K \setminus I})\beta_{yu.ak \setminus i }^T  + \frac{\sigma_{yy.ak \setminus i u}}{(n - |K \setminus I| - 3)\sigma_{aa.k \setminus i}}, \\
        &= \text{MSE}(\hat \tau_{K \setminus I }),      
\end{align}
and this concludes the proof.
\end{proof}

To conclude the proof of Theorem~2, we recall that precision variables, extended confounding variables and irrelevant variables form a complete partitioning of all covariates:
\begin{align*}
    \mathcal{V} \setminus \{A, Y\} &= \mathcal{P} \cup \mathcal{W} \cup \mathcal{I}
\end{align*}
The intersection of each MSE-optimal adjustment set $O_n^i \in \mathcal{O}_n(\mathcal{M}, \hat \tau_{K})$ with all covariates $\mathcal{V} \setminus \{A, Y\}$ is:
\begin{align*}
    O^i_n \cap \mathcal{V}\setminus \{A, Y\} &= O^i_n \cap \left( \mathcal{P} \cup \mathcal{W} \cup \mathcal{I} \right)  \\
    &= (O^i_n \cap \mathcal{P}) \cup (O^i_n \cap \mathcal{W}) \cup (O^i_n \cap \mathcal{I}).
\end{align*}
By Lemma~\ref{lemma:irrelevant}, we know that $O^i_n \cap \mathcal{I} = \emptyset$. We proceed by expanding the sets of precision variables $\mathcal{P}$ and extended confounding variables $\mathcal{W}$ into their suboptimal and non-suboptimal subsets, i.e. $\mathcal{P} = \mathcal{S}^P \cup (\mathcal{P} \setminus \mathcal{S}^P)$. It follows:
\begin{align*}
     O^i_n \cap \mathcal{V}\setminus \{A, Y\} &= \left[O^i_n \cap \left\{\mathcal{S}^P \cup (\mathcal{P} \setminus \mathcal{S}^P)\right\}\right] \cup \left[O^i_n \cap \left\{\mathcal{S}^W \cup (\mathcal{W} \setminus \mathcal{S}^W)\right\}\right] \\
     &= \{  (O^i_n \cap \mathcal{S}^P) \cup (O^i_n \cap \mathcal{P}\setminus \mathcal{S}^P) \} \cup \{ (O^i_n \cap \mathcal{S}^W) \cup (O^i_n \cap \mathcal{W}\setminus \mathcal{S}^W) \}
\end{align*}
Let $\mathcal{O}^*_n(\mathcal{M}, \hat \tau_K)$ be the set of MSE-optimal adjustment sets, such that $O^{*, i}_n \cap \mathcal{S}^P = \emptyset$ for all $O^{*, i}_n \in \mathcal{O}^*_n(\mathcal{M}, \hat \tau_K)$. By Lemma~\ref{lemma:precision}, we know that such an MSE-optimal adjustment set exists, i.e. $\mathcal{O}^*_n(\mathcal{M}, \hat \tau_K)$ is non-empty. By Lemma~\ref{lemma:confounding}, for any adjustment set $K$ that includes a suboptimal confounding variable $S$, there exists a non-suboptimal confounding variable $W^*$, such that $K \setminus S \cup W^*$ yields a lower or equal mean squared error than $K$. It follows that $\mathcal{O}^*_n(\mathcal{M}, \hat \tau_K)$ includes at least one MSE-optimal adjustment set $O^*_n$, such that $O^{*}_n \cap \mathcal{S}^W = \emptyset$. Hence, there exists an MSE-optimal adjustment set $O^*_n$, such that
\begin{align*}
    O^*_n \cap \mathcal{V}\setminus \{A, Y\} 
    &= (O^*_n \cap \mathcal{P}\setminus \mathcal{S}^P) \cup  (O^*_n \cap \mathcal{W}\setminus \mathcal{S}^W) \\
    &= O^*_n \cap  (\mathcal{P} \setminus \mathcal{S}^P \cup \mathcal{W} \setminus \mathcal{S}^W ) \\
    &\subseteq  \mathcal{P} \setminus \mathcal{S}^P \cup \mathcal{W} \setminus \mathcal{S}^W.
\end{align*}

\subsection*{Proof of Theorem 3 (Forbidden combinations)}
To prove this, we show that if the d-separation $L_i \indep_{\mathcal{G}'} Y \mid   L_{-i} \cup K$ holds in $\mathcal{G}'=\mathcal{G} \setminus (A \rightarrow Y)$ for any $K \subseteq \mathcal{V} \setminus \{A, Y\}$, any adjustment set $X$ with $L \subseteq X$ provides a higher mean squared error than the adjustment set $X \setminus L_i$. For this, we can follow Equations~\eqref{eq:irrelevant-first}\textendash\eqref{eq:irrelevant-last} in the proof of Lemma~\ref{lemma:irrelevant} for irrelevant variables, assuming that $K = X \setminus L_i$ and $I = L_i$. The variable $L_i$ can be considered an irrelevant variable given $L_{-i}$ in the sense that the d-separation properties for irrelevant variables in Lemma~\ref{lemma-d-separation-I} hold for any $Z$ where $L_{-i} \subseteq Z$, based on the d-separation $L_i \indep_{\mathcal{G}'} Y \mid   L_{-i} \cup K$ for any $K \subseteq \mathcal{V} \setminus \{A, Y\}$. Hence, we can follow Lemma~\ref{lemma:irrelevant} to show that the adjustment set $K \setminus I = X \setminus L_{-i} $ provides a lower mean squared error than the adjustment set $K \cup I = X$ if $L \subseteq X$, and we conclude that $ X$ can not be an MSE-optimal adjustment set if $L \subseteq X$.

\subsection*{Proof of Theorem 4 (Suboptimal valid adjustment sets)}
We know that both $K$ and $O$ are valid adjustment sets, which means they yield zero bias. According to Theorem 3 in \cite{henckelGraphicalCriteriaEfficient2022a}, the optimal adjustment set $O$ provides an asymptotic variance that is lower than or equal to the asymptotic variance provided by any other valid adjustment set, i.e.
\begin{equation}
    \mathrm{aVar}(\hat \tau_{O}) \leq \mathrm{aVar}(\hat \tau_{K}).
\end{equation}
Using our assumption that $|K| \geq |O|$, we can show that the mean squared error yielded by $O$ is smaller than or equal to the mean squared error yielded by $K$:
\begin{align*}
     \rm{MSE}(\hat \tau_O) &= B^2(\hat \tau_O) + \frac{\mathrm{aVar}(\hat \tau_{O})}{n-|O|-3} \\
    &= B^2(\hat \tau_K) + \frac{\mathrm{aVar}(\hat \tau_{O})}{n-|O|-3} \\
    &\leq B^2(\hat \tau_K) + \frac{\mathrm{aVar}(\hat \tau_{K})}{n-|O|-3} \\
    &\leq B^2(\hat \tau_K) + \frac{\mathrm{aVar}(\hat \tau_{K})}{n-|K|-3} \\
    &= \rm{MSE}(\hat \tau_K).
\end{align*}

\section*{Appendix 3}
\label{A:examples}
\subsection*{Examples}
In the following, we show how our graphical criteria can be applied to the graphs in Figure~1 and 2 from the main paper. First, we copy the graphs here for convenience.

\begin{figure}[h]
    \centering
       \begin{subfigure}[b]{0.33\textwidth}
        \centering
        \begin{tikzpicture}[
            font=\fontsize{8}{10}\selectfont,
            node distance=1.3cm,
            on grid,
            auto,
            block/.style={circle, draw, inner sep=0pt, outer sep=0pt, minimum size=0.5cm}
        ]

        \node [block] (A) {A};
        \node [block, above right=of A] (V2) {$W_2$};
        \node [block, right=of A, below right=of V2] (Y) {$Y$};
        \node [block, above right=of V2] (O1) {$O_1$};
        \node [block, above left=of V2] (V1) {$W_1$};         
        \node [block, below right=of A] (O2) {$O_2$};
        
        \draw[-{Latex[length=2mm]}, blue] (A) -- (Y) node[midway, above, black] {}; 
        \draw[-{Latex[length=2mm]}] (V2) -- (O1) node[midway, right, pos=0.1] {};          
        \draw[-{Latex[length=2mm]}] (V1) -- (O1)node[midway, above]{};                   
        \draw[-{Latex[length=2mm]}] (O1) -- (Y) node[midway, right] {};                  
        \draw[-{Latex[length=2mm]}] (V1) -- (A) node[midway, left] {};                   
        \draw[-{Latex[length=2mm]}] (V2) -- (A) node[midway, left, pos=0.1] {};           
        \draw[-{Latex[length=2mm]}] (O2) -- (A) node[midway, left, pos=0.2] {};            
        \draw[-{Latex[length=2mm]}] (O2) -- (Y) node[midway, right, pos=0.2] {};         
                
\end{tikzpicture}
        \caption{$\mathcal{G}_1$}
        \label{a:fig:g1}
    \end{subfigure}%
    \begin{subfigure}[b]{0.33\textwidth}
    \centering
        \begin{tikzpicture}[
            font=\fontsize{8}{10}\selectfont,
            node distance=1.3cm,
            on grid,
            auto,
            block/.style={circle, draw, inner sep=0pt, outer sep=0pt, minimum size=0.5cm}
        ]
        
        \node [block] (A) {A};
        \node [block, above right=of A] (M) {$C_1$};
        \node [block, below right=of M] (Y) {$Y$};
        \node [block, above=of Y, above right=of M] (O1) {$O_1$};
        \node [block, above=of A, above left=of M] (I1) {$W_1$};
        \node [block, below right=of A] (O2) {$O_2$};
        
        \draw[-{Latex[length=2mm]}, blue] (A) -- (Y) node[midway, above, black] {};  
        \draw[-{Latex[length=2mm]}] (O1) -- (M) node[midway, left, pos=0.1] {};              
        \draw[-{Latex[length=2mm]}] (I1) -- (M) node[midway, right, pos=0.1] {};          
        \draw[-{Latex[length=2mm]}] (O1) -- (Y) node[midway, right] {};                  
        \draw[-{Latex[length=2mm]}] (I1) -- (A) node[midway, left] {};                    
        \draw[-{Latex[length=2mm]}] (O2) -- (A) node[midway, left, pos=0.2] {};           
        \draw[-{Latex[length=2mm]}] (O2) -- (Y) node[midway, right, pos=0.2] {};         
                
\end{tikzpicture}
        \caption{$\mathcal{G}_2$}
        \label{a:fig:g2}
    \end{subfigure}
    \begin{subfigure}[b]{0.33\textwidth}
           \centering
           \begin{tikzpicture}[
        font=\fontsize{8}{10}\selectfont,
        node distance=1.3cm,
        on grid,
        auto,
        block/.style={circle, draw, inner sep=0pt, outer sep=0pt, minimum size=0.5cm}
    ]
    
    \node [block] (A) {A};
    \node [block, right=of A] (Y) {$Y$};
    \node [block, above=of Y] (O1) {$O_1$};
    \node [block, above=of A] (I1) {$S_1$};
    \node [block, right=of Y] (P1) {$O_2$};
    \node [block, above=of P1] (I2) {$S_2$};
    \node [block, right=of P1] (I3) {$S_3$};
    \node [block, below=of P1] (P) {$P_1$};
    \node [block, left=of P] (O3) {$O_3$};
    \node [block, right=of P] (O4) {$O_4$};
    \node [block, left=of O3] (IV1) {$I_1$};
    
    \draw[-{Latex[length=2mm]}, blue] (A) -- (Y) node[midway, above, black] {$\tau$};
    \draw[-{Latex[length=2mm]}] (I1) -- (O1) node[midway, right, pos=0.1] {};
    \draw[-{Latex[length=2mm]}] (O1) -- (Y) node[midway, right] {};
    \draw[-{Latex[length=2mm]}] (I1) -- (A) node[midway, left] {};    
    \draw[-{Latex[length=2mm]}] (P1) -- (Y) node[midway, left] {};   
    \draw[-{Latex[length=2mm]}] (I2) -- (P1) node[midway, left] {};    
    \draw[-{Latex[length=2mm]}] (I3) -- (P1) node[midway, left] {};  
    \draw[-{Latex[length=2mm]}] (P) -- (O3) node[midway, left] {};  
    \draw[-{Latex[length=2mm]}] (O3) -- (Y) node[midway, left] {};  
    \draw[-{Latex[length=2mm]}] (P) -- (O4) node[midway, left] {}; 
    \draw[-{Latex[length=2mm]}] (O4) -- (Y) node[midway, left] {}; 
    \draw[-{Latex[length=2mm]}] (IV1) -- (A) node[midway, left] {}; 
    \end{tikzpicture}
    \caption{$\mathcal{G}_3$}
    \label{a:fig:g3}
    \end{subfigure}
    \label{a:fig:examples}
\end{figure}
In $\mathcal{G}_1$, we have $2^4= 16$ possible subsets of $\mathcal{V} \setminus \{A, Y\} = \{W_1, W_2, O_1, O_2\}$. With Theorem~3, we can identify $\{W_1, O_1\}$ and $\{W_2, O_1\}$ as forbidden combinations and can remove any adjustment set $Z$ with $\{W_1, O_1\} \subseteq Z$ or $\{W_2, O_1\} \subseteq Z$, i.e. the adjustment sets $\{W_1, O_1\}$, $\{W_1, O_1, W_2\}$, $\{W_1, O_1, O_2\}$, $\{W_1, O_1, W_2, O_2\}$, $\{W_2, O_1\}$, $\{W_2, O_1, O_2\}$. Using Theorem~4, we can also remove the valid adjustment set $\{W_1, W_2, O_2\}$. We end up with a search space consisting of 9 possible adjustment sets, namely $\{\}$, $\{O_1\}$, $\{O_2\}$, $\{W_1, O_2\}$, $\{W_2, O_2\}$, $\{W_1\}$, $\{W_2\}$, $\{W_1, W_2\}$, $\{O_1, O_2\}$.

In $\mathcal{G}_2$, we also have $2^4= 16$ possible subsets of $\mathcal{V} \setminus \{A, Y\} = \{W_1, C_1, O_1, O_2\}$. With Definition~7, we can identify $W_1$ as a suboptimal confounding variable. With Theorem~2, we can remove every adjustment set $Z$ with $W_1 \in Z$ from the search space, resulting in $2^3=8$ remaining sets. Using Theorem~3, we can identify $\{C_1, O_1\}$ as a forbidden combination, and remove the adjustment sets $\{C_1, O_1\}$ and $\{C_1, O_1, O_2\}$. We end up with a search space consisting of 6 possible adjustment sets, namely $\{\}$, $\{O_1\}$, $\{O_2\}$, $\{C_1\}$, $\{C_1, O_2\}$, $\{O_1, O_2\}$.

In $\mathcal{G}_3$, we have $2^9= 512$ possible subsets of $\mathcal{V} \setminus \{A, Y\}$. With Definition~7, Definition~5 and Definition~\ref{def:irrelevant}, we can identify $S_1$ as a suboptimal confounding variable, $S_2$ and $S_3$ as suboptimal precision variables and $I_1$ as irrelevant variable respectively. With Theorem~2, we can remove every adjustment set $Z$ containing $S_1$, $S_2$, $S_3$ or $I_1$ from the search space, resulting in $2^5=32$ remaining sets. Using Theorem~3, we can identify $\{P_1, O_3, O_4\}$ as a forbidden combination, and remove the adjustment sets $\{P_1, O_3, O_4\}$, $\{P_1, O_3, O_4, O_1\}$, $\{P_1, O_3, O_4, O_2\}$ and $\{P_1, O_3, O_4, O_1, O_2\}$. We end up with a search space consisting of $32 - 4 = 28$ adjustment sets.

\label{A:algorithm}
\label{app:results}
\section*{Appendix 4}
\subsection*{Algorithm}
First, we prune the variables to exclude suboptimal precision variables, suboptimal confounding variables and irrelevant variables with Theorem~2. Then, we prune the power set of the remaining candidate variables $\mathbb{P}(\text{candidates})$ with Theorem~3 and Theorem~4. The remaining variable sets form our search space for the MSE-optimal adjustment set.

For each potential MSE-optimal adjustment set $K$, we estimate the variance of the ordinary least squares estimator $\hat \tau_{K}$ from Equation~\eqref{eq:ols-var-1}, i.e. $\mathrm{var}(\hat \tau_{K}) = \sigma_{yy.ak} / \mathrm{RSS}_{a.k}$, with $\hat \sigma_{yy.ak} = \mathrm{RSS}_{y.ak} / (n - |K| - 1)$. If $\mathrm{var}(\hat \tau_{K})$ is larger or equal to the variance of the optimal adjustment set $O$, we can discard $K$, as it can not yield a better mean squared error than $O$. To estimate the bias yielded by all remaining adjustment sets, we use 1000 bootstrap resamples of the data to average the difference between the estimate of $\hat \tau_{K}$ and the unbiased estimate of $\hat \tau_{O}$. 

\begin{algorithm}
\caption{Estimating Treatment Effect with MSE-Optimal Adjustment Set}
\KwIn{Directed Acyclic Graph $\mathcal{G}=(\mathcal{V}, \mathcal{E})$, Data $\mathcal{D}$}
\KwOut{Estimated Treatment Effect $\hat{\tau}$}
Initialize $best\_set \gets O(\mathcal{G})$\;
Initialize $o\_variance \gets \text{EstimateVariance}(\mathcal{D}, O(\mathcal{G}))$\;
Initialize $min\_mse \gets o\_variance$\;
$candidates \gets \text{PruneVariables}(\mathcal{G})$\;
$adjustment\_sets \gets \mathbb{P}(candidates)$\;
$\mathcal{Z} \gets \text{PruneCombinations}(\mathcal{G}, adjustment\_sets)$\;
\For{$Z \in \mathcal{Z}$}{
    $variance \gets \text{EstimateVariance}(\mathcal{D}, Z)$\;
    \If{$variance < o\_variance$}{
        $bias \gets \text{EstimateBias}(\mathcal{D}, Z, O(\mathcal{G}))$\;
        $mse \gets bias^2 + variance$\;
        \If{$mse < min\_mse$}{
            $min\_mse \gets mse$\;
            $best\_set \gets Z$\;
        }
    }
}
$estimated\_effect \gets \text{OrdinaryLeastSquares}(\mathcal{D}, best\_set)$\;
\Return $estimated\_effect$\;
\label{algo}
\end{algorithm}

\subsection*{Experiments}

\begin{table}[htbp]
\centering
\footnotesize            

\caption{Comparison of $O$ and $\hat O_n(\mathcal{M}_2, \hat \tau_{K})$ with $\mathcal{M}_2$ from Figure~1 (right) in the main paper, 10\,000 random seeds.}
\label{tab:results}

\begin{tabular}{@{}lcc@{}}
\toprule
Sample size &
\multicolumn{1}{c}{$O$ (Mean $\pm$ SD)} &
\multicolumn{1}{c}{$\hat O_n$ (Mean $\pm$ SD)} \\ \midrule
10   & 0.1486 (0.2599) & \textbf{0.1412 (0.2906)} \\
20   & 0.0511 (0.0814) & \textbf{0.0477 (0.0779)} \\
30   & 0.0303 (0.0468) & \textbf{0.0280 (0.0435)} \\
40   & 0.0216 (0.0321) & \textbf{0.0205 (0.0312)} \\
50   & 0.0169 (0.0246) & \textbf{0.0161 (0.0240)} \\
100  & 0.0082 (0.0116) & \textbf{0.0079 (0.0111)} \\
150  & \textbf{0.0053 (0.0077)} & \textbf{0.0053 (0.0072)} \\
200  & \textbf{0.0040 (0.0056)} & 0.0041 (0.0056) \\
1000 & \textbf{0.0008 (0.0011)} & 0.0009 (0.0012) \\ \bottomrule
\end{tabular}
\end{table}

\begin{table}[htbp]
\centering
\footnotesize

\caption{Mean squared error (MSE) of the adjustment set with the smallest estimated variance, data sampled from causal model $\mathcal{M}_1$ in Figure~1 (left) in the main paper, 10\,000 random seeds.  Results from Table~1 in the main paper are included for comparison.}
\label{tab:results2}

\begin{tabular}{@{}lccc@{}}
\toprule
Sample size &
\multicolumn{1}{c}{MSE (Mean $\pm$ SD)} &
\multicolumn{1}{c}{MSE for $O$ (Mean $\pm$ SD)} &
\multicolumn{1}{c}{MSE for $\hat O_n$ (Mean $\pm$ SD)} \\ \midrule
10   & \textbf{0.0003 (0.0004)} & 0.1234 (0.2379) & 0.0926 (0.2343) \\
20   & \textbf{0.0002 (0.0002)} & 0.0403 (0.0622) & 0.0306 (0.0673) \\
30   & \textbf{0.0002 (0.0001)} & 0.0247 (0.0373) & 0.0182 (0.0374) \\
40   & \textbf{0.0002 (0.0001)} & 0.0172 (0.0252) & 0.0135 (0.0270) \\
50   & \textbf{0.0002 (0.0001)} & 0.0136 (0.0199) & 0.0103 (0.0201) \\
100  & \textbf{0.0002 (0.0001)} & 0.0064 (0.0091) & 0.0048 (0.0095) \\
500  & \textbf{0.0002 (0.0000)} & 0.0012 (0.0017) & 0.0010 (0.0018) \\
1000 & \textbf{0.0002 (0.0000)} & 0.0006 (0.0009) & 0.0005 (0.0009) \\ \bottomrule
\end{tabular}
\end{table}

\end{document}